\documentclass[journal, twoside]{IEEEtran}

\ifCLASSINFOpdf
\else
\fi

\usepackage{graphicx}
\usepackage{array, amssymb, amscd, amsthm}
\usepackage[cmex10]{amsmath}
\usepackage{algorithmic}
\usepackage[tight,footnotesize]{subfigure}
\usepackage{stfloats}
\usepackage{url}
\usepackage{cite} 
\usepackage[geometry]{ifsym}
\usepackage{bbm}
\usepackage{algorithm}
\usepackage{algorithmic}
\usepackage{color}
\usepackage{makecell}

\usepackage{multirow}

\allowdisplaybreaks

\let\oldFootnote\footnote
\newcommand\nextToken\relax

\renewcommand\footnote[1]{%
    \oldFootnote{#1}\futurelet\nextToken\isFootnote}

\newcommand\isFootnote{%
    \ifx\footnote\nextToken\textsuperscript{,}\fi}

\newtheorem{assumption}{Assumption}
\newtheorem{proposition}{Proposition}
\newtheorem{lemma}{Lemma}
\newtheorem{thm}{Theorem}
\newtheorem{cor}{Corollary}

\newtheorem{remark}{Remark}
\newtheorem{term}{Term}

\DeclareMathOperator{\co}{\text{co}}
\DeclareMathOperator{\dist}{\text{dist}}

\newcommand{\sg}[1]{{\color{black}{#1}}}
\newcommand{\maria}[1]{{\color{black}{#1}}}

\makeatletter
\newenvironment{proofof}[1]{\par
  \pushQED{\qed}%
  \normalfont \topsep6\p@\@plus6\p@\relax
  \trivlist
  \item[\hskip\labelsep
        \bfseries
    Proof of #1\@addpunct{.}]\ignorespaces
}{%
  \popQED\endtrivlist\@endpefalse
}
\makeatother

\begin{document}

\title{Distributed constrained optimization and consensus in uncertain networks via proximal minimization}

\author{Kostas~Margellos, Alessandro~Falsone, Simone~Garatti and Maria~Prandini
\thanks{Research was supported by the European Commission, H2020, under the project UnCoVerCPS, grant number 643921. Preliminary results, related to Sections II-IV of the current manuscript, can be found in \cite{ACC_paper_2016}.}\vspace{-1\baselineskip}
\thanks{K. Margellos is with the Department of Engineering Science, University of Oxford,
Parks Road, OX1 3PJ, Oxford, UK, e-mail: \texttt{kostas.margellos@eng.ox.ac.uk}

A. Falsone, S. Garatti and M. Prandini are with the Dipartimento di Elettronica Informazione e Bioingegneria, Politecnico di Milano,
Piazza Leonardo da Vinci 32, 20133 Milano, Italy, e-mail: \texttt{\{alessandro.falsone, simone.garatti, maria.prandini\}@polimi.it}}}

\maketitle
\IEEEpeerreviewmaketitle

\begin{abstract}
We provide a unifying framework for distributed convex optimization over time-varying networks, in the presence of constraints and uncertainty, features that are typically treated separately in the literature. We adopt a proximal minimization perspective and show that this set-up allows us to bypass the difficulties of existing algorithms while simplifying the underlying mathematical analysis. We develop an iterative algorithm and show convergence of the resulting scheme to some optimizer of the centralized problem.
To deal with the case where the agents' constraint sets are affected by a possibly common uncertainty vector, we follow a scenario-based methodology and offer probabilistic guarantees regarding the feasibility properties of the resulting solution. To this end, we provide a distributed implementation of the scenario approach, allowing agents to use a different set of uncertainty scenarios in their local optimization programs. The efficacy of our algorithm is demonstrated by means of a numerical example related to a regression problem subject to regularization.
\end{abstract}


\section{Introduction} \label{sec:secI}
\IEEEPARstart{O}{ptimization} in multi-agent networks has attracted significant attention in the control and signal processing literature, due to its applicability in different domains like power systems \cite{Bolognani_etal_2015,Zhang_Giannakis_2015}, wireless networks \cite{Mateos_Giannakis_2012,Baingana_etal_2014}, robotics \cite{Martinez_etal_2007}, etc. Typically, agents solve a local decision making problem, communicate their decisions with other agents and repeat the process on the basis of the new information received. The main objective of this cooperative set-up is for agents to agree on a common decision that optimizes a certain performance criterion for the overall multi-agent system while satisfying local constraints. This distributed optimization \sg{scheme} leads to computational and \sg{communication} savings compared to centralized paradigms, \sg{while allowing agents to keep privacy by exchanging partial information only.}

\subsection{Contributions of this work}
In this paper we deal with distributed convex optimization problems over time-varying networks, under a possibly different constraint set per agent, and in the presence of uncertainty. Focusing first on the deterministic case, we construct an iterative, proximal minimization based algorithm. Proximal minimization, \sg{where a penalty term (proxy) is introduced in the objective function of each agents' local decision problem,} serves as an alternative to (sub)gradient methods. This is interesting per se, since it
constitutes the multi-agent counterpart of connections between proximal algorithms and gradient methods that have been established in the literature for single-agent problems (see \cite{Bertsekas_Tsitsiklis_1997}).
Moreover, as observed in \cite{Bertsekas_2011} with reference to incremental algorithms, the proximal minimization approach  leads to numerically more stable algorithms compared to their gradient-based counterparts. \\
A rigorous and detailed analysis \sg{is provided, showing} that the proposed iterative scheme converges to an optimizer of the centralized problem counterpart. This is achieved without imposing differentiability assumptions or requiring excessive memory capabilities as other methods in the literature (see Section \ref{sec:secIB} for a detailed review).

\sg{We move then to} the case where constraints depend on an uncertain parameter and should be robustly satisfied for all values that this parameter may take. This poses additional challenges when devising a distributed solution methodology. Here, \sg{we exploit} results on scenario-based optimization \cite{Calafiore_Campi_2006,Campi_Garatti_2008,Campi_etal_2009,Campi_Garatti_2011,Garatti_Campi_2013,compression_paper_2015}. In particular, we assume that each agent is provided with \maria{its own, different from the other agents,} set of uncertainty realizations (scenarios) and enforces the constraints \sg{corresponding to these scenarios only}.
We then show that our distributed algorithm is applicable and that the converged solution is feasible in a probabilistic sense for the constraints of the centralized problem, i.e., it satisfies with high probability all agents' constraints when an unseen uncertainty instance is realized. To achieve this we rely on the novel contribution of \cite{new_scenario_paper_2015}, which leads to a sharper result compared to the one that would be obtained by a direct application of the basic scenario theory \cite{Campi_Garatti_2008}.
Our approach can be thought of as the data driven counterpart of robust or worst-case optimization paradigms, enabling us to provide a priori guarantees on the probability of constraint satisfaction without imposing any assumptions on the underlying distribution of the uncertainty and its moments, and/or the geometry of the uncertainty sets (e.g., \cite[Chapters 6, 7]{Boyd_2004}); however, providing the overall feasibility statement with a certain confidence. The proposed distributed implementation of the scenario approach, which is instead typically performed in a centralized fashion, \sg{allows for a reduction of the communication burden and the satisfaction of privacy requirements regarding the available knowledge on the uncertain parameter.}

\subsection{Related work} \label{sec:secIB}
Most literature builds on the seminal work of \cite{Tsitsiklis_1984,Tsitsiklis_etal_1986,Bertsekas_Tsitsiklis_1997} (see
also \cite{Boyd_etal_2010,Necoara_etal_2011} and references therein for a more recent problem exposition), where a wide range of \sg{decentralized} optimization problems is considered, using techniques based on gradient descent, dual decomposition, and the method of multipliers. The recent work of \cite{Grammatico_etal_2015} deals with similar problems but from a game theoretic perspective.

Distributed optimization problems, in the absence of constraints \sg{though}, have been considered in \cite{Vicsek_etal_1995,Jadbabaie_etal_2003,Olfati_Saber_etal_2004,Cao_etal_2005,Boyd_etal_2006,Johansson_2007,Nedic_Ozdaglar_2009,Olshevsky_Tsitsiklis_2011,Nedic_Olshevsky_2015,Varagnolo_etal_2015}. In most of these references the underlying network is allowed to be time-varying. In the presence of constraints, the authors of \cite{Johansson_2008,Necoara_Suykens_2009,Necoara_etal_2011,Pu_etal_2015} adopt Newton-based or gradient/subgradient-based approaches and show asymptotic agreement of the agents' solutions to an optimizer of the centralized problem, in \cite{Wei_Ozdaglar_2015,Bianchi_etal_2016} a distributed alternating direction method of multipliers approach is adopted and its convergence properties are analyzed, whereas in \cite{Carlone_etal_2014,Burger_etal_2014} a constraints consensus approach is adopted. In these \sg{contributions}, however, the underlying network is time-invariant, while agents are required to have certain memory capabilities. In a time-varying environment, \sg{as that considered in the present paper,} \cite{Nedic_etal_2010,Zhu_Martinez_2012} \sg{propose} a projected subgradient methodology to solve distributed convex optimization problems in the presence of constraints. In \cite{Nedic_etal_2010}, however, the particular case where the agents' constraint sets are all identical is considered. As a result, the computational complexity of each agents' local optimization program is the same \sg{as that} of the \sg{centralized algorithm}.
Our approach, \sg{which allows for different constraint sets per agent,} is most closely related to the work of  \cite{Zhu_Martinez_2012},  
but we adopt a proximal minimization instead of a subgradient-based perspective, thus avoiding the requirement for gradient/subgradient computation.

In most of the aforementioned references a deterministic set-up is considered.  Results taking into account both constraints and uncertainty have recently appeared in \cite{Lee_Nedic_2013,Towfic_Sayed_2014,Lee_Nedic_2015}.
In \cite{Towfic_Sayed_2014} a penalty-based approach is adopted and convergence of the proposed scheme is shown under the assumption that the algorithm is initialized with some feasible solution, which, however, \sg{can be difficult} to compute. This is not required in the approach proposed in this paper. In \cite{Lee_Nedic_2015}
an asynchronous algorithm is developed for a \sg{quite} particular communication protocol that involves gossiping, i.e., pairwise communication, under stronger regularity conditions (strong convexity of the agents' objective function).

Our set-up is closely related, albeit different from the approach of \cite{Lee_Nedic_2013}, \sg{which proposes a projected gradient descent approach where at every iteration a random extraction of each agents' constraints is performed}. In \cite{Lee_Nedic_2013} \sg{almost sure convergence is proved, but this requires that} different scenarios \sg{are} extracted at every iteration, and these scenarios \sg{must} be independent from each other, and independent across iterations. This creates difficulties in accounting for temporal correlation of the uncertain parameter, and poses challenges if sampling from the underlying distribution is computationally expensive. On the contrary, \sg{in our algorithm each agent is provided with a given number of scenarios (which accounts for data driven optimization too) and the same uncertainty scenarios are used at every iteration}.
\sg{In this case}, convergence in \cite{Lee_Nedic_2013} is not guaranteed, whilst our scenario-based approach provides probabilistic feasibility, \sg{as opposed to almost sure feasibility,} guarantees. \sg{This probabilistic treatment of uncertainty, which is particularly suited to data based optimization, does not appear, to the best of our knowledge, in any of the aforementioned references}.
Moreover, \sg{differently from \cite{Lee_Nedic_2013},} our proximal minimization perspective allows us to bypass the requirement for gradient computations, rendering the developed programs amenable to existing numerical solvers, and do not impose differentiability assumptions on the agents' objective functions and Lipschitz continuity of the objective gradients.

Finally, \sg{it is perhaps worth mentioning that} our approach is fundamentally different from the randomized algorithm of \cite{Carlone_etal_2014}, which is based on iteratively exchanging active constraints over a time-invariant network; in our case the network is time-varying and we do not require for constraint exchange, thus reducing the communication requirements.

\maria{For a quick overview, 
Table \ref{table:Lit} provides a classification of the literature most closely related to our work in terms of communication requirements (which is related to whether the underlying network is time-varying or not) and their ability to deal with different types of constraints (which is also related to the overall computational effort as explained before).}

\begin{footnotesize}
\begin {table}[t]
\begin{center}
\begin{tabular}[t]{| l | c | c | c |  }
 \hline
\multicolumn{2}{|c|}{\diaghead{\theadfont Diag ColumnmnHead II} {Agents' \\ decision vectors}{\\Network} }
& time-invariant & {time-varying }\\
 \hline \hline
\multicolumn{2}{| l |}{unconstrained} & {-}    &{ \cite{Tsitsiklis_1984,Tsitsiklis_etal_1986}, \cite{Jadbabaie_etal_2003,Olfati_Saber_etal_2004,Cao_etal_2005,Boyd_etal_2006,Johansson_2007,Nedic_Ozdaglar_2009,Olshevsky_Tsitsiklis_2011,Nedic_Olshevsky_2015,Varagnolo_etal_2015} } \\ \hline
\multicolumn{2}{| l |}{same constraints} &  { - }   & {\cite{Nedic_etal_2010} }\\
 \hline
	{different } & deterministic & { \cite{Johansson_2008,Necoara_Suykens_2009,Necoara_etal_2011,Pu_etal_2015,Wei_Ozdaglar_2015,Bianchi_etal_2016,Carlone_etal_2014,Burger_etal_2014,Nedic_etal_2010} } & \cite{Zhu_Martinez_2012} \\ \cline{2-4}
	constraints & uncertain & -& \cite{Lee_Nedic_2013}, our work\\
	\hline
\end{tabular}
\end{center}
\caption{Classification of related work.}
\label{table:Lit}
\end{table}
\end{footnotesize}

All the aforementioned references, and our work as well, are concerned with static optimization problems, or problems with discrete time dynamics.
As for distributed optimization for continuous time systems, the interested reader is referred to \cite{Cortes_Bullo_2005,Cortes_Bullo_2009,Wang_Elia_2011,Kia_2011,Kia_etal_2015,Cherukuri_Cortes_2015}, and references therein.

\subsection{Structure of the paper}
The paper unfolds as follows: In Section \ref{sec:SecII} we provide a formal statement of the problem under study, and, focusing on the deterministic case, formulate the proposed distributed algorithm based on proximal minimization; \sg{convergence and optimality are also discussed, but to streamline the presentation all} proofs, \sg{along with some preparatory results and useful relations regarding the agents' local solutions,} are deferred to Section \ref{sec:SecIII}. Section \ref{sec:SecV} deals with the stochastic case where constraints are affected by uncertainty, following a scenario-based methodology. To illustrate the efficacy of our algorithm, Section \ref{sec:exam} provides a distributed implementation of a regression problem subject to $L1$-regularization. Finally, Section \ref{sec:secVI} concludes the paper and provides some directions for future work.
%

\subsection*{Notation} \label{sec:Notation}
$\mathbb{R}$, $\mathbb{R}_+$ denote the real and positive real numbers, and $\mathbb{N}$, $\mathbb{N}_+$ the natural and positive natural numbers, respectively.
For any $x \in \mathbb{R}^n$, $\|x\|$ denotes the Euclidean norm of $x$, whereas for a scalar $a \in \mathbb{R}$, $|a|$ denotes its absolute value. Moreover, $x^\top$ denotes the transpose vector of $x$. For a continuously differentiable function $f(\cdot):~ \mathbb{R}^n \rightarrow \mathbb{R}$, $\nabla f(x)$ is the gradient of $f(x)$. Given a set $X$, we denote by $\co(X)$ its convex hull. We write $\dist(y,X)$ to denote the Euclidean distance of a vector $y$ from a set $X$, i.e., $\dist(y,X) = \inf_{x \in X} \|y-x\|$.

A vector $a \in \mathbb{R}^m$ is said to be a stochastic vector if all its components $a_j$ are non-negative and sum up to one, i.e., $\sum_{j=1}^m a_j = 1$. Consider a square matrix $A \in \mathbb{R}^{m \times m}$ and denote its $i$-th column by $a^i \in \mathbb{R}^m$, $i=1,\ldots,m$. $A$ is said to be doubly stochastic if both its rows and columns are stochastic vectors, i.e., $\sum_{i=1}^m a^i_j = 1$ for all $j=1,\ldots,m$, and $\sum_{j=1}^m a_j^i = 1$ for all $i=1,\ldots,m$.


\section{Distributed constrained convex optimization} \label{sec:SecII}
\subsection{Problem set-up} \label{sec:SecIIA}
We consider a time-varying network of $m$ agents that communicate to cooperatively solve an optimization problem of the form

\begin{align}
\mathcal{P}_{\delta}:~& \min_{x \in \mathbb{R}^{n}} \sum_{i=1}^m f_i(x) \label{eq:P_con_delta} \\
& \text{subject to } x \in \bigcap_{\delta \in \Delta} \bigcap_{i=1}^m X_i(\delta), \nonumber
\end{align}
where $x \in \mathbb{R}^n$ represents a vector of $n$ decision variables, and $\delta \in \Delta$. We assume that $\Delta$ is endowed with a $\sigma$-algebra $\mathcal{D}$ and that $\mathbb{P}$ is a fixed, but possibly unknown, probability measure defined over $\mathcal{D}$.
For each $i=1,\ldots,m$, $f_i(\cdot):~ \mathbb{R}^n \rightarrow \mathbb{R}$ is the objective function of agent $i$, whereas, for any $\delta \in \Delta$, $X_i(\delta) \subseteq \mathbb{R}^n$ is its constraint set\footnote{For any $\delta \in \Delta$, $X_i(\delta)$ is supposed to represent all constraints to the decision vector imposed by agent $i$, including explicit constraints expressed e.g., by inequalities like $h_i(x,\delta) \leq 0$ and restrictions to the domain of the objective function $f_i$.}.

Problem $\mathcal{P}_{\delta}$ is a robust program, where any feasible solution $x$ should belong to $\bigcap_{i=1}^m X_i(\delta)$ for all realizations $\delta \in \Delta$ of the uncertainty. Note that the fact that uncertainty appears only in the constraints and not in the objective functions is without loss of generality; in the opposite case, an epigraphic reformulation would recast the problem in the form of $\mathcal{P}_{\delta}$.

\sg{Due to the presence of uncertainty, problem $\mathcal{P}_{\delta}$ may be very difficult to solve, especially when $\Delta$ is a continuous set. Hence, a proper way to deal with uncertainty must be introduced. Moreover, our perspective is that $f_i(\cdot)$ and $X_i$ represent private information, available only to agent $i$ and/or even though the whole information were available to all agents, imposing all the constraints in one shot, would result in a computationally intensive program. This motivates the use of a distributed algorithm.}

\sg{To ease the exposition of our distributed algorithm,} we focus first on the \sg{following} deterministic variant of $\mathcal{P}_{\delta}$ \sg{with constraint sets being independent of $\delta$:}
\begin{align}
\mathcal{P}:~& \min_{x \in \mathbb{R}^n} \sum_{i=1}^m f_i(x) \label{eq:P_con} \\
& \text{subject to } x \in \bigcap_{i=1}^m X_i. \nonumber
\end{align}
\sg{$\mathcal{P}_{\delta}$ will be revisited in Section \ref{sec:SecV} where we will specify how to deal with the presence of uncertainty.}

Since most of the subsequent results are based on $f_i(\cdot)$ and $X_i$ being convex, we formalize it in the following assumption.
\begin{assumption} \label{ass:Convex}
[Convexity] For each $i=1,\ldots,m$, the function $f_i(\cdot):~ \mathbb{R}^n \rightarrow \mathbb{R}$ and the set $X_i \subseteq \mathbb{R}^n$ are convex.
\end{assumption}

\subsection{A new proximal minimization-based algorithm}

\sg{The pseudo-code of the proposed proximal minimization-based} iterative approach \sg{is given} in Algorithm \ref{alg:Alg1}.
\begin{algorithm}[t]
\caption{Distributed algorithm}
\begin{algorithmic}[1]
\STATE \textbf{Initialization} \\
\STATE ~~~~Set $\{a_j^i(k)\}_{k \geq 0}$, for all $i,j = 1,\ldots,m$. \\
\STATE ~~~~Set $\{c(k)\}_{k \geq 0}$. \\
\STATE ~~~~$k=0$. \\
\STATE ~~~~Consider $x_i(0) \in X_i$, for all $i=1,\ldots,m$. \\
\STATE \textbf{For $i=1,\ldots,m$ repeat until convergence} \\
\STATE ~~~~$z_i(k) = \sum_{j=1}^m a_j^i(k) x_j(k)$. \\
\STATE ~~~~~$x_i(k+1) = \arg \min_{x_i \in X_i} f_i(x_i) + \frac{1}{2 c(k)} \|z_i(k) - x_i\|^2$. \\
\STATE ~~~~$k \leftarrow k+1$.
\end{algorithmic}
\label{alg:Alg1}
\end{algorithm}
Initially, each agent $i$, $i=1,\ldots,m$, starts with some tentative value $x_i(0)$ which belongs to the local constraint set $X_i$ of agent $i$, but not necessarily to $\bigcap_{i=1}^m X_i$.
One sensible choice for $x_i(0)$ is to set it such that $x_i(0) \in \arg \min_{x_i \in X_i} f_i(x_i)$.
At iteration $k$, each agent $i$ constructs a weighted average $z_i(k)$ of the solutions communicated by the other agents and its local one (step 7, Algorithm \ref{alg:Alg1}, \sg{where $a_j^i(k)$ are the weights}). \sg{Then, each agent} solves a local minimization problem, involving its local objective function $f_i(x_i)$ and a quadratic term, penalizing the difference from $z_i(k)$ (step 8, Algorithm \ref{alg:Alg1}, \sg{where the coefficient $c(k)$, which is assumed to be non-increasing with $k$, regulates the relative importance of the two terms}). Note that, unlike $\mathcal{P}$, under Assumption \ref{ass:Convex} and due to the presence of the quadratic penalty term, the resulting problem is strictly convex with respect to $x_i$, and hence admits a unique solution.

For each $k \geq 0$ the information exchange between the $m$ agents can be represented by a directed graph $(V,E_k)$, where the nodes $V = \{1,\ldots,m\}$ are the agents and the set $E_k$ of directed edges \maria{$(j,i)$} \sg{indicating that at time $k$ agent $i$ receives information from agent $j$} is given by
\begin{align}
E_k = \big \{ (j,i):~ a_j^i(k) > 0 \big \}. \label{eq:edges}
\end{align}
From \eqref{eq:edges}, we set $a_j^i(k) = 0$ in the absence of communication. If \maria{$(j,i) \in E_k$} we say that $j$ is a neighboring agent of $i$ at time $k$. Under this set-up, Algorithm \ref{alg:Alg1} provides a fully distributed implementation, where at iteration $k$ each agent $i=1,\ldots,m$ receives information only from neighboring agents. \sg{Moreover, this information exchange is time-varying and may be occasionally} \maria{absent.} \sg{However, the following connectivity and communication assumption is made, where} $E_{\infty} = \big \{ (j,i):~ (j,i) \in E_k \text{ for infinitely many } k \big \}$ \sg{denotes} the set of edges $(j,i)$ \sg{representing} agent pairs that communicate directly infinitely often.
\begin{assumption} \label{ass:Network}
[Connectivity and Communication] The graph $(V,E_{\infty})$ is strongly connected, i.e., for any two nodes there exists a path of directed edges that connects them. Moreover, there exists $T \geq 1$ such that for every $(j,i) \in E_{\infty}$, agent $i$ receives information from a neighboring agent $j$ at least once every consecutive $T$ iterations.
\end{assumption}

Assumption \ref{ass:Network} guarantees that any pair of agents communicates directly infinitely often, and the intercommunication interval is bounded. For further details on the interpretation of the imposed network structure the reader is referred to \cite{Nedic_Ozdaglar_2009,Nedic_etal_2010}.

Algorithm \ref{alg:Alg1} terminates if the iterates maintained by all agents converge. From an implementation point of view, agent $i$, $i=1,\ldots,m$, will terminate its update process if the absolute difference (relative difference can also be used) between two consecutive iterates $\| x_i(k+1) - x_i(k)\|$ keeps below some user-defined tolerance for a number of iterations equal to $T$ (see Assumption \ref{ass:Network}) times the diameter of the graph (i.e., the greatest distance between any pair of nodes connected via an edge in $E_{\infty}$). This is the worst case number of iterations required for an agent to communicate with all others in the network; note that if an agent terminated the process at the first iteration where the desired tolerance is met, then convergence would not be guaranteed since its solution may still change as an effect of other agents updating their solutions.

The proposed iterative methodology resembles the structure of proximal minimization for constrained convex optimization \cite[Chapter 3.4.3]{Bertsekas_Tsitsiklis_1997}.
The difference, however, is that our set-up is distributed and the quadratic term in step 8 does not penalize the deviation of $x_i$ from the previous iterate $x_i(k)$, but from an appropriately weighted average $z_i(k)$. Note that, in contrast with the inspiring work in \cite{Nedic_etal_2010,Zhu_Martinez_2012,Lee_Nedic_2013} addressing $\mathcal{P}$ under a similar set-up but following a projected subgradient approach, our proximal minimization-based approach allows for an intuitive economic interpretation: at every iteration $k$ we penalize a consensus residual proxy by the time-varying coefficient $1/(2 c(k))$, which progressively increases. This can be thought of as a pricing settling mechanism, where the more we delay to achieve consensus the higher the price is.

In the case where $a_j^i(k) = 1/m$ for all $i,j = 1,\ldots,m$, for all $k \geq 0$, that corresponds to a decentralized control paradigm, the solution of our proximal minimization approach coincides with the one obtained when the alternating direction of multipliers \cite{Bertsekas_Tsitsiklis_1997}, \cite{Boyd_etal_2010}, is applied to this problem (see eq. (4.72)-(4.74), p. 254 in \cite{Bertsekas_Tsitsiklis_1997}). In the latter the quadratic penalty term is not added to the local objective function as in step 8 of Algorithm \ref{alg:Alg1}, but to the Lagrangian function of an equivalent problem, and the coefficient $c(k)$ is an arbitrary constant independent of $k$; however, a dual-update step is required. Formal connections between penalty methods and the method of multipliers have been established in \cite{Bertsekas_1976}.

\begin{remark}[Application to a specific problem structure]
Algorithm \ref{alg:Alg1} can be simplified when the underlying optimization problem exhibits a specific structure, namely agents need to agree on a common decision vector $y \in \mathbb{R}^{\bar{n}}$, but each of them decides upon a local decision vector $u_i \in \mathbb{R}^{n_i}$, $i=1,\ldots,m$ as well:
\begin{align}
& \min_{y \in \mathbb{R}^{\bar{n}},\{u_i \in \mathbb{R}^{n_i}\}_{i=1}^m} \sum_{i=1}^m f_i(y,u_i) \nonumber \\
& \text{subject to }  y \in \bigcap_{i=1}^m Y_i, 
\ u_i \in U_i, \ i=1,\ldots,m, \label{eq:P_new_con2}
\end{align}
where $Y_i \in \mathbb{R}^{\bar{n}}$ and $U_i \subseteq \mathbb{R}^{n_i}$, for all $i=1,\ldots,m$.
Provided that Assumptions \ref{ass:Convex}-\ref{ass:Network} hold for problem \eqref{eq:P_new_con2} with $x=(y, u_1, \dots,u_m)$ and $X_i=Y_i\times\mathbb{R}^{n_1}\times \dots \times U_i \times \dots \times \mathbb{R}^{n_m}$, we can rewrite it as $\min_{y \in \mathbb{R}^{\bar{n}}} \sum_{i=1}^m g_i(y)$ subject to $y \in \bigcap_{i=1}^m Y_i$, where $g_i(y) = \min_{u_i \in U_i} f_i(y,u_i)$ and simplify Algorithm \ref{alg:Alg1} by replacing steps 7-8 with:
\begin{align*}
& z_i(k) = \sum_{j=1}^m a_j^i(k) y_j(k), \\
& \big ( y_i(k+1), u_i(k+1) \big ) \nonumber \\
&~~~~~~~~~~~~= \arg\min_{y_i \in Y_i, u_i \in U_i} f_i(y_i,u_i) + \frac{1}{2 c(k)} \|z_i(k) - y_i\|^2. 
\end{align*}
This entails that agents only need to communicate their local estimates $y_i(k)$, $i=1,\ldots,m$, of the common decision vector $y$ while the local solutions related to $u_i$, $i=1,\ldots,m$, need not be exchanged.
\end{remark}

\subsection{Further structural assumptions and communication requirements} \label{sec:SecIIB}

We impose some additional assumptions on the structure of problem $\mathcal{P}$ in \eqref{eq:P_con} and the communication set-up that is considered in this paper. These assumptions will play a crucial role in the \sg{proof of convergence} of Section \ref{sec:SecIII}.

\begin{assumption} \label{ass:CompactLip}
[Compactness] For each $i=1,\ldots,m$, $X_i \subseteq \mathbb{R}^n$ is compact.
\end{assumption}

Note that due to Assumption \ref{ass:CompactLip}, $\co \big ( \bigcup_{i=1}^m X_i \big )$ is also compact. Let then $D \in \mathbb{R}_+$ be such that $\|x\| \leq D$ for all $x \in \co \big ( \bigcup_{i=1}^m X_i \big )$. Moreover, due to Assumptions \ref{ass:Convex} and \ref{ass:CompactLip}, $f_i(\cdot):~ \mathbb{R}^n \rightarrow \mathbb{R}$ is Lipschitz continuous on $X_i$ with Lipschitz constant $L_i \in \mathbb{R}_+$, i.e., for all $i=1,\ldots,m$,
\begin{align}
|f_i(x) - f_i(y)| \leq L_i \|x-y\|, \text{ for all } x,y \in X_i. \label{eq:Lipschitz}
\end{align}

\begin{assumption} \label{ass:SlaterPoint}
[Interior point] The feasibility region $\bigcap_{i=1}^m X_i$ of $\mathcal{P}$ has a non-empty interior, i.e., there exists $\bar{x} \in \bigcap_{i=1}^m X_i$ and $\rho \in \mathbb{R}_+$ such that $\{x \in \mathbb{R}^n:~ \|x-\bar{x}\| < \rho \} \subset \bigcap_{i=1}^m X_i$.
\end{assumption}

Due to Assumption \ref{ass:SlaterPoint}, by the Weierstrass' theorem (Proposition A.8, p. 625 in \cite{Bertsekas_Tsitsiklis_1997}), $\mathcal{P}$ admits at least one optimal solution. Therefore, if we denote by $X^* \subseteq \bigcap_{i=1}^m X_i$ the set of optimizers of $\mathcal{P}$, then $X^*$ is non-empty. Notice also that $f_i(\cdot)$, $i=1,\ldots,m$, is continuous due to the convexity condition of Assumption \ref{ass:Convex}; the addition of Assumption \ref{ass:CompactLip} is to imply Lipschitz continuity. However, $f_i(\cdot)$, $i=1,\ldots,m$, is not required to be differentiable.

We impose the following assumption on the coefficients $\{c(k)\}_{k \geq 0}$, that appear in step 8 of Algorithm \ref{alg:Alg1}.
\begin{assumption} \label{ass:ConvCoef}
[Coefficient $\{c(k)\}_{k \geq 0}$] Assume that for all $k \geq 0$, $c(k) \in \mathbb{R}_+$ and $\{c(k)\}_{k \geq 0}$ is a non-increasing sequence, i.e., $c(k) \leq c(r)$ for all $k \geq r$, with $r \geq 0$. Moreover,
\begin{enumerate}
\item $\sum_{k=0}^{\infty} c(k) = \infty$,
\item $\sum_{k=0}^{\infty} c(k)^2 < \infty$.
\end{enumerate}
\end{assumption}

In standard proximal minimization \cite{Bertsekas_Tsitsiklis_1997} convergence is highly dependent on the appropriate choice of $c(k)$.
Assumption \ref{ass:ConvCoef} is in fact needed to guarantee convergence of Algorithm \ref{alg:Alg1}. A direct consequence of the last part of Assumption \ref{ass:ConvCoef} is that $\lim_{k \rightarrow \infty} c(k) = 0$. One choice for $\{c(k)\}_{k \geq 0}$ that satisfies the conditions of Assumption \ref{ass:ConvCoef} is to select it from the class of generalized harmonic series, e.g., $c(k) = \alpha/(k+1)$ for some $\alpha \in \mathbb{R}_+$.
Note that Assumption \ref{ass:ConvCoef} is in a sense analogous to the conditions that the authors of \cite{Nedic_etal_2010,Zhu_Martinez_2012} impose on the step-size of their subgradient algorithm. It should be also noted that our set-up is synchronous, using the same $c(k)$ for all agents, at every iteration $k$. Extension to an asynchronous implementation is a topic for future work.

In line with \cite{Tsitsiklis_1984,Tsitsiklis_etal_1986,Olshevsky_Tsitsiklis_2011} we impose the following assumptions on the information exchange between the agents.

\begin{assumption} \label{ass:Weights}
[Weight coefficients] There exists $\eta \in (0,1)$ such that for all $i,j \in \{1,\ldots,m\}$ and all $k \geq 0$, $a_j^i(k) \in \mathbb{R}_+ \cup \{0\}$, $a_i^i(k) \geq \eta$, and $a_j^i(k) > 0$ implies that $a_j^i(k) \geq \eta$.
Moreover, for all $k \geq 0$,
\begin{enumerate}
\item $\sum_{j=1}^{m} a_j^i(k) = 1$ for all $i=1,\ldots,m$,
\item $\sum_{i=1}^{m} a_j^i(k) = 1$ for all $j=1,\ldots,m$.
\end{enumerate}
\end{assumption}

Assumptions \ref{ass:Network}  and \ref{ass:Weights} are identical to Assumptions 2 and 5 in \cite{Nedic_etal_2010} (the same assumptions are also imposed in \cite{Zhu_Martinez_2012}), but were reported also here to ease the reader and facilitate the exposition of our results.
Note that these are rather standard for distributed optimization and consensus problems; for possible relaxations the reader is referred to \cite{Nedic_etal_2009,Olshevsky_Tsitsiklis_2011}.
The interpretation of having a uniform lower bound $\eta$, independent of $k$, for the coefficients $a_j^i(k)$ in Assumption \ref{ass:Weights} is that it ensures that each agent is mixing information received by other agents at a non-diminishing rate in time \cite{Nedic_etal_2010}. Moreover, \sg{points 1) and 2) in} Assumption \ref{ass:Weights} ensure that this mixing is a convex combination of the other agent estimates, assigning a non-zero weight to its local one \sg{since} $a_i^i(k) \geq \eta$. Note that satisfying Assumption \ref{ass:Weights} requires agents to agree on an infinite sequence of doubly stochastic matrices (double stochasticity arises due to conditions 1 and 2 in Assumption \ref{ass:Weights}), where $a_j^i(k)$ would be element $(i,j)$ of the matrix at iteration $k$. This agreement should be performed prior to the execution of the algorithm in a centralized manner, and the resulting matrices have to be communicated to all agents via some consensus scheme; this is standard in distributed optimization algorithms of this type (see also \cite{Nedic_etal_2010,Olshevsky_Tsitsiklis_2011,Zhu_Martinez_2012}). It would be of interest to construct doubly stochastic matrices in a distributed manner using the machinery of \cite{Sinkhorn_Knopp_1967}; however, exploiting these results requires further investigation and is outside the scope of the paper.

\subsection{Statement of the main convergence result} \label{sec:SecIIC}
Under the structural assumptions and the communication set-up imposed in the previous subsection, Algorithm \ref{alg:Alg1} converges and agents reach consensus, in the sense that their local estimates $x_i(k)$, $i=1,\ldots,m$, converge to some minimizer of problem $\mathcal{P}$.
This is formally stated in the following theorem, which constitutes \sg{one} main contribution of our paper.

\begin{thm}\label{thm:optimality}
Consider Assumptions \ref{ass:Convex}-\ref{ass:Weights} and Algorithm \ref{alg:Alg1}. We have that, for some minimizer $x^* \in X^*$ of $\mathcal{P}$,
\begin{align}
\lim_{k \rightarrow \infty} \|x_i(k) - x^*\| = 0, \text{ for all } i=1,\ldots,m. \label{eq:optimality}
\end{align}
\end{thm}

\sg{To streamline the contribution of the paper,} the \sg{rather technical} proof of this statement is deferred to Section \ref{sec:SecIVC}.

\section{Dealing with uncertainty} \label{sec:SecV}

In this section, we revisit problem $\mathcal{P}_{\delta}$ in \eqref{eq:P_con_delta}, and \sg{give} a methodology to deal with the presence of uncertainty. Motivated by data driven considerations, we assume that each agent $i$, $i=1,\ldots,m$, is provided with a
fixed number of realizations of $\delta$, referred to as scenarios, extracted according to the underlying probability measure $\mathbb{P}$ with which $\delta$ takes values in $\Delta$.
According to the information about the scenarios that agents possess, two cases are distinguished in the sequel \sg{(scenarios as a common resource vs. scenarios as a private resource)} and the properties of the corresponding scenario programs are analyzed.

\sg{Throughout,} the following modifications to Assumptions \ref{ass:Convex}-\ref{ass:SlaterPoint}
are imposed: 1) For each $i=1,\ldots,m$, $X_i(\delta)$ is a convex set for any $\delta \in \Delta$.
2) For each $i=1,\ldots,m$, and for any finite set $S$ of values for $\delta$, $\bigcap_{\delta \in S} X_i(\delta)$ is compact.
3) For any finite set $S$ of values for $\delta$, $\bigcap_{i=1}^m \bigcap_{\delta \in S} X_i(\delta)$ has a non-empty interior.

For the subsequent analysis, note that for any $N \in \mathbb{N}_+$, $\mathbb{P}^N$ denotes the corresponding product measure. We assume measurability of all involved functions and sets.

\subsection{Probabilistic feasibility - Scenarios as a common resource} \label{sec:SecVB}
We first consider the case where all agents are provided with the same scenarios of $\delta$, i.e., scenarios can be thought of as a common resource for the agents. This is the case if all agents have access to the same set of historical data for $\delta$, or if agents communicate the scenarios with each other. The latter case, however, increases the communication requirements.

Let $\bar{N} \in \mathbb{N}_+$ denote the number of scenarios, and $\bar{S} = \{\delta^{(1)}, \ldots, \delta^{(\bar{N})} \} \subset \Delta$ be the set of scenarios available to all agents. The scenarios are independently and identically distributed (i.i.d.) according to $\mathbb{P}$. Consider then the following optimization program $\mathcal{P}_{\bar{N}}$, where the subscript $\bar{N}$ is introduced to emphasize the dependency with respect to the uncertainty scenarios.
\begin{align}
\mathcal{P}_{\bar{N}}:~& \min_{x \in \mathbb{R}^n} \sum_{i=1}^m f_i(x) \nonumber \\
& \text{subject to } x \in \bigcap_{\delta \in \bar{S}} \bigcap_{i=1}^m X_i(\delta). \label{eq:P_con_N_bar}
\end{align}
Clearly, $x \in \bigcap_{\delta \in \bar{S}} \bigcap_{i=1}^m X_i(\delta)$ is equivalent to $x \in \bigcap_{i=1}^m \bigcap_{\delta \in \bar{S}} X_i(\delta)$, and $\mathcal{P}_{\bar{N}}$ is amenable to be solved via the distributed algorithm of Section \ref{sec:SecIIA}. In fact, one can apply Algorithm \ref{alg:Alg1} with $\bigcap_{\delta \in \bar{S}} X_i(\delta)$ in place of $X_i$, for all $i=1,\ldots,m$.
Let $X_{\bar{N}}^* \subseteq \bigcap_{i=1}^m \bigcap_{\delta \in \bar{S}} X_i(\delta)$ be the set of minimizers of $\mathcal{P}_{\bar{N}}$. We then have the following corollary of Theorem \ref{thm:optimality}.

\begin{cor} \label{cor:optimality_N_bar}
Consider Assumptions \ref{ass:Convex}-\ref{ass:Weights} with the modifications stated in Section \ref{sec:SecV}, and Algorithm \ref{alg:Alg1}. We have that, for some $x_{\bar{N}}^* \in X_{\bar{N}}^*$,
\begin{align}
\lim_{k \rightarrow \infty} \|x_{i,\bar{N}}(k) - x_{\bar{N}}^*\| = 0, \text{ for all } i=1,\ldots,m, \label{eq:optimality_N}
\end{align}
where $x_{i,\bar{N}}(k)$ denotes the solution generated at iteration $k$, step 8 of Algorithm \ref{alg:Alg1}, when $X_i$ is replaced by $\bigcap_{\delta \in \bar{S}} X_i(\delta)$.
\end{cor}

We address the problem of quantifying the robustness of the minimizer $x_{\bar{N}}^*$ of $\mathcal{P}_{\bar{N}}$ to which our iterative scheme converges according to Corollary \ref{cor:optimality_N_bar}. In the current set-up a complete answer is given by the scenario approach theory \cite{Calafiore_Campi_2006,Campi_Garatti_2008}, which shows that $x_{\bar{N}}^*$ is feasible for $\mathcal{P}_{\delta}$ up to a quantifiable level $\bar{\varepsilon}$.
This result is based on the notion of support constraints (see also Definition 4 in \cite{Calafiore_Campi_2006}), and in particular on the notion of support set \cite{new_scenario_paper_2015} (also referred to as compression scheme in \cite{compression_paper_2015}). Given an optimization program, we say that a subset of the constraints constitutes a support set, if it is the minimal cardinality subset of the constraints such that by solving the optimization problem considering only this \sg{subset} of constraints, we obtain the same solution \sg{to} the original problem \sg{where all the constraints are enforced}. As
a consequence, all constraints that do not belong to the support set are in a sense redundant since their removal leaves the optimal solution unaffected.

By Theorem 3 of \cite{Calafiore_Campi_2006}, for any convex optimization program the cardinality of the support set is at most equal to the number of decision variables $n$, whereas in \cite{Schildbach_etal_2013} a refined bound is provided. The subsequent result is valid for any given bound on the cardinality of the support set. Therefore, and since $\mathcal{P}_{\bar{N}}$ is convex, let $d \in \mathbb{N}_+$ be a known upper-bound for the cardinality of its support set.
A direct application of the scenario approach theory in \cite{Calafiore_Campi_2006} leads then to the following result.

\begin{thm} \label{thm:prob_feas_centr}
Fix $\beta \in (0,1)$ and let
\begin{align}
\bar{\varepsilon} = 1 - \sqrt[\bar{N}-d]{\frac{\beta}{{\bar{N} \choose d}}}. \label{eq:eps_explicit_bar}
\end{align}
We then have that
\begin{align}
&\mathbb{P}^{\bar{N}} \Big \{ \bar{S} \in \Delta^{\bar{N}} :~ \nonumber \\
&\mathbb{P} \Big \{ \delta \in \Delta :~ x_{\bar{N}}^* \notin \bigcap_{i=1}^m X_i(\delta) \Big \}  \leq \bar{\varepsilon} \Big \} \geq 1-\beta. \label{eq:prob_all_cons_bar}
\end{align}
\end{thm}

In words, Theorem \ref{thm:prob_feas_centr} implies that with confidence at least $1-\beta$, $x_{\bar{N}}^*$ is feasible for $\mathcal{P}_\delta$ apart from a set of uncertainty instances with measure at most $\bar{\varepsilon}$. Notice that $\bar{\varepsilon}$ is in fact a function of $\bar{N}$, $\beta$ and $d$. We suppress this dependency though to simplify notation.
Note that even though $\mathcal{P}_{\bar{N}}$ does not necessarily have a unique solution, Theorem \ref{thm:prob_feas_centr} still holds for the solution returned by Algorithm \ref{alg:Alg1} (assuming convergence), since it is a deterministic algorithm and hence serves as a tie-break rule to select among the possibly multiple minimizers.

Following \cite{Campi_Garatti_2008}, \eqref{eq:eps_explicit_bar} could be replaced with an improved $\bar{\varepsilon}$, obtained as the solution of $\sum_{k=0}^{d-1} {\bar{N} \choose k} \bar{\varepsilon}^k \big ( 1- \bar{\varepsilon} \big )^{\bar{N} - k} = \beta$. However, we use \eqref{eq:eps_explicit_bar} since it gives an explicit relation expression for $\bar{\varepsilon}$, and also renders \eqref{eq:prob_all_cons_bar} directly comparable with the results provided in the next subsection.

In case $\bar{\varepsilon}$ exceeds one, the result becomes trivial. However, note that Theorem \ref{thm:prob_feas_centr} can be also reversed (as in experiment design) to compute the number $\bar{N}$ of scenarios that is required for \eqref{eq:prob_all_cons_bar} to hold for given $\bar{\varepsilon}, \beta \in (0,1)$. This can be determined by solving \eqref{eq:eps_explicit_bar} with respect to $\bar{N}$ with the chosen $\bar{\varepsilon}$ fixed (e.g., using numerical inversion). The reader is referred to Theorem 1 of \cite{Calafiore_Campi_2006} for an explicit expression of $\bar{N}$.

\subsection{Probabilistic feasibility - Scenarios as a private resource} \label{sec:SecVC}
We now consider the case where the information carried by the scenarios is distributed, that is, each agent has its own set of scenarios, which constitute agents' private information.
\sg{Specifically,} assume that each agent $i$, $i=1,\ldots,m$, is provided with a set $S_i = \{\delta_i^{(1)}, \ldots, \delta_i^{(N_i)} \} \subset \Delta$ of $N_i \in \mathbb{N}_+$ i.i.d. scenarios of $\delta$, extracted according to the underlying probability measure $\mathbb{P}$. Here, $\delta_i^{(j)}$ denotes scenario $j$ of agent $i$, $j=1,\ldots,N_i$, $i=1,\ldots,m$. The scenarios across the different sets $S_i$, $i=1,\ldots,m$, are independent from each other. The total number of scenarios is $N = \sum_{i=1}^m N_i$. Consider then the following optimization program $\mathcal{P}_N$, where each agent has its own scenario set.
\begin{align}
\mathcal{P}_{N}:~& \min_{x \in \mathbb{R}^n} \sum_{i=1}^m f_i(x) \nonumber \\
& \text{subject to } x \in \bigcap_{i=1}^m \bigcap_{\delta \in S_i} X_i(\delta). \label{eq:P_con_N}
\end{align}
Program $\mathcal{P}_{N}$ can be solved via the distributed algorithm of Section \ref{sec:SecIIA}, so that a solution is obtained without exchanging any private information regarding the scenarios.
In fact, one can apply Algorithm \ref{alg:Alg1} with $\bigcap_{\delta \in S_i} X_i(\delta)$ in place of $X_i$, for all $i=1,\ldots,m$.

Similarly to Corollary \ref{cor:optimality_N_bar}, letting $X_N^* \subseteq \bigcap_{i=1}^m \bigcap_{\delta \in S_i} X_i(\delta)$ be the set of minimizers of $\mathcal{P}_N$, we have the following corollary of Theorem \ref{thm:optimality}.

\begin{cor} \label{cor:optimality_N}
Consider Assumptions \ref{ass:Convex}-\ref{ass:Weights} with the modifications stated in Section \ref{sec:SecV}, and Algorithm \ref{alg:Alg1}. We have that, for some $x_N^* \in X_N^*$,
\begin{align}
\lim_{k \rightarrow \infty} \|x_{i,N}(k) - x_N^*\| = 0, \text{ for all } i=1,\ldots,m, \label{eq:optimality_N}
\end{align}
where $x_{i,N}(k)$ denotes the solution generated at iteration $k$, step 8 of Algorithm \ref{alg:Alg1}, when $X_i$ is replaced by $\bigcap_{\delta \in S_i} X_i(\delta)$.
\end{cor}

As in Section \ref{sec:SecVB}, we show that the minimizer $x_{N}^*$ of $\mathcal{P}_{N}$ to which our iterative scheme converges according to Corollary \ref{cor:optimality_N} is feasible in a probabilistic sense for $\mathcal{P}_\delta$. Here, a difficulty arises, since we seek to quantify the probability that $x_N^*$ satisfies the global constraint $\bigcap_{i=1}^m X_i(\delta)$, where $\delta$ is a common parameter to all $X_i(\delta)$, $i=1,\ldots,m$, while $x_N^*$ has been computed considering $X_i(\delta)$ for uncertainty scenarios that are independent from those of $X_j(\delta)$, $j \neq i$, $i=1,\ldots,m$.

Let $S = \{S_i\}_{i=1}^m$ be a collection of the scenarios of all agents. Similarly to the previous case, we denote by $d \in \mathbb{N}_+$ a known upper-bound for the cardinality of the support set of $\mathcal{P}_N$. However, the way the constraints of \sg{the support} set are split among the agents depends on the specific $S$ employed. Therefore, for each set of scenarios $S$ \sg{and for $i=1,\ldots,m$}, denote by $d_{i,N} ( S ) \in \mathbb{N}$ (possibly equal to zero) the number of constraints that belong to \sg{both} the support set of $\mathcal{P}_N$ and $S_i$, i.e., the constraints of agent $i$. We then have that $\sum_{i=1}^m d_{i,N} ( S ) \leq d$, for any $S \in \Delta^N$. For short we will write $d_{i,N}$ instead of $d_{i,N} (S)$ and make the dependency on $S$ explicit only when necessary.

\subsubsection{A naive result} \label{sec:SecVB1}
For any collection of agents' scenarios, it clearly holds that $d_{i,N} \leq d$ for all $i=1,\ldots,m$, for any scenario set.
Thus, for each $i=1,\ldots,m$, Theorem \ref{thm:prob_feas_centr} can be applied conditionally to the scenarios of all other agents to obtain a local, in the sense that it holds only for the constraints of agent $i$, feasibility characterization. Fix $\beta_i \in (0,1)$ and let
\begin{align}
\widetilde{\varepsilon}_i = 1 - \sqrt[N_i-d]{\frac{\beta_i}{{N_i \choose d}}}. \label{eq:eps_explicit_tilde}
\end{align}
We then have that
\begin{align}
&\mathbb{P}^{N} \Big \{ S \in \Delta^{N} :~ \mathbb{P} \Big \{ \delta \in \Delta :~ x_{N}^* \notin X_i(\delta) \Big \}  \leq \widetilde{\varepsilon}_i \Big \} \geq 1-\beta_i. \label{eq:prob_agent_i}
\end{align}

By the subadditivity of $\mathbb{P}^N$ and $\mathbb{P}$, \eqref{eq:prob_agent_i} can be used to quantify the probabilistic feasibility of
$x_N^*$ with respect to the global constraint $\bigcap_{i=1}^m X_i(\delta)$. Following the proof of Corollary 1 in \cite{Kariotoglou_etal_2015}, where a similar argument is provided, we have that
\begin{align}
&\mathbb{P}^{N} \Big \{ S \in \Delta^{N} :~ \mathbb{P} \Big \{ \delta \in \Delta :~ x_{N}^* \notin \bigcap_{i=1}^m X_i(\delta) \Big \}  \leq \sum_{i=1}^m \widetilde{\varepsilon}_i \Big \} \nonumber \\
&= \mathbb{P}^{N} \Big \{ S \in \Delta^{N} :~ 
\mathbb{P} \Big \{ \delta \in \Delta :~ \exists i \in \{1,\ldots,m\}, x_{N}^* \notin X_i(\delta) \Big \}  \nonumber \\
&~~~~ ~~~~ \leq \sum_{i=1}^m \widetilde{\varepsilon}_i \Big \} \nonumber \\
&= \mathbb{P}^{N} \Big \{ S \in \Delta^{N} :~ \mathbb{P} \Big \{ \bigcup_{i=1}^m \Big \{ \delta \in \Delta :~ x_{N}^* \notin X_i(\delta) \Big \}  \Big \}  \leq \sum_{i=1}^m \widetilde{\varepsilon}_i \Big \} \nonumber \\
&\geq \mathbb{P}^{N} \Big \{ S \in \Delta^{N} :~ \sum_{i=1}^m \mathbb{P} \Big \{ \delta \in \Delta :~ x_{N}^* \notin X_i(\delta) \Big \}  \leq \sum_{i=1}^m \widetilde{\varepsilon}_i \Big \} \nonumber \\
&\geq \mathbb{P}^{N} \Big \{ \bigcap_{i=1}^m \Big \{ S \in \Delta^{N} :~ \mathbb{P} \Big \{ \delta \in \Delta :~ x_{N}^* \notin X_i(\delta) \Big \}  \leq \widetilde{\varepsilon}_i \Big \} \Big \} \nonumber \\
&\geq 1 - \sum_{i=1}^m \mathbb{P}^{N} \Big \{ S \in \Delta^{N} :~ \mathbb{P} \Big \{ \delta \in \Delta :~ x_{N}^* \notin X_i(\delta) \Big \}  > \widetilde{\varepsilon}_i \Big \} \nonumber \\
&\geq 1 - \sum_{i=1}^m \beta_i,
\label{eq:prob_agent_sub}
\end{align}
which leads to the following proposition.

\begin{proposition} \label{prob:prob_feas_sub}
Fix $\beta \in (0,1)$ and choose $\beta_i$, $i=1,\ldots,m$, such that $\sum_{i=1}^m \beta_i = \beta$. For each $i=1,\ldots,m$, let $\widetilde{\varepsilon}_i$ be as in \eqref{eq:eps_explicit_tilde} and set $\widetilde{\varepsilon} = \sum_{i=1}^m \widetilde{\varepsilon}_i$. We then have that
\begin{align}
&\mathbb{P}^N \Big \{ S \in \Delta^N :~ \mathbb{P} \Big \{ \delta \in \Delta :~ x_{N}^* \notin \bigcap_{i=1}^m X_i(\delta) \Big \} \leq \widetilde{\varepsilon} \Big \} \geq 1-\beta. \label{eq:prob_all_cons}
\end{align}
\end{proposition}

Proposition \ref{prob:prob_feas_sub} implies that with confidence at least $1-\beta$, $x_N^*$ is feasible for $\mathcal{P}_\delta$ apart from a set with measure at most $\widetilde{\varepsilon}$. This result, however, tends to be very conservative thus prohibiting its applicability to problems with a high number of agents. This can be seen by comparing $\widetilde{\varepsilon}$ with $\bar{\varepsilon}$, where the latter corresponds to the case where scenarios are treated as a common resource. To this end, consider the particular set-up where $N_i = \bar{N}$ and $\beta_i = \beta/m$, for all $i=1,\ldots,m$. By \eqref{eq:eps_explicit_bar} and \eqref{eq:eps_explicit_tilde}, it follows that $\widetilde{\varepsilon} = m \widetilde{\varepsilon}_i \approx m \bar{\varepsilon}$, thus growing approximately (we do not have exact equality since $\beta_i = \beta/m$) linearly with the number of agents. This can be also observed in the numerical comparison of Section \ref{sec:SecVB2} (see Fig. \ref{fig:guarantees}).
The issue with Proposition \ref{prob:prob_feas_sub} is that it accounts for a worst-case setting, where $d_{i,N} = d$ for all $i=1,\ldots,m$; however, this can not occur, since $\sum_{i=1}^m d_{i,N} \leq d$ implies that if $d_{i,N} = d$ for some $i$, then $d_{j,N} = 0$, for all $j \neq i$, $i=1,\ldots,m$.

\subsubsection{A tighter result} \label{sec:SecVB2}
To alleviate the conservatism of Proposition \ref{prob:prob_feas_sub}, and exploit the fact that $\sum_{i=1}^m d_{i,N} \leq d$, we use the recent results of \cite{new_scenario_paper_2015}.

For each $i=1,\ldots,m$, fix $\beta_i \in (0,1)$ and consider a function $\varepsilon_i(\cdot)$ defined as follows:
\begin{align}
\varepsilon_i(k) = 1 - \sqrt[N_i-k]{\frac{\beta_i}{(d+1) {N_i \choose k}}}, \text{ for all } k=0,\ldots,d. \label{eq:eps_explicit}
\end{align}
Notice that $\varepsilon_i(\cdot)$ is also a function of $N_i$, $\beta_i$ and $d$, but this dependency is suppressed to simplify notation.
For each $i=1,\ldots,m$, working conditionally with respect to the scenarios $S \setminus S_i$ of all other agents, Theorem 1 of \cite{new_scenario_paper_2015} entails that
\begin{align}
&\mathbb{P}^N \Big \{ S \in \Delta^N :~ \mathbb{P} \Big \{ \delta \in \Delta :~ x_N^* \notin X_i(\delta) \Big \} \leq \varepsilon_i (d_{i,N} ) \nonumber \\
& ~~~~~~~~~~~~~~~~~\Big |~ \big \{ S \setminus S_i \in \Delta^{N-N_i} \big \} \Big \}  \geq 1-\beta_i. \label{eq:prob_agent_cond}
\end{align}
Integrating \eqref{eq:prob_agent_cond} with respect to the probability of realizing the scenarios $S \setminus S_i$, we have that
\begin{align}
&\mathbb{P}^N \Big \{ S \in \Delta^N :~\mathbb{P} \Big \{ \delta \in \Delta :~ x_N^* \notin X_i(\delta) \Big \} \leq \varepsilon_i (d_{i,N}) \Big \} \nonumber \\
& ~~~~~\geq 1-\beta_i. \label{eq:prob_agent}
\end{align}
The statement in \eqref{eq:prob_agent} implies that for each agent $i=1,\ldots,m$, with confidence at least $1-\beta_i$, the probability that $x_N^*$ does not belong to the constraint set $X_i(\delta)$ of agent $i$ is at most equal to $\varepsilon_i (d_{i,N})$.

Note, however, that \eqref{eq:prob_agent} is very different from \eqref{eq:prob_agent_i}, which is obtained by means of the basic scenario approach theory, since $d_{i,N}$ is not known a-priori but depends on the extracted scenarios.
Using \eqref{eq:prob_agent} in place of \eqref{eq:prob_agent_i} in the the derivations of \eqref{eq:prob_agent_sub}, by the subadditivity of $\mathbb{P}^N$ and $\mathbb{P}$, we have that
\begin{align}
&\mathbb{P}^N \Big \{ S \in \Delta^N :~ \mathbb{P} \Big \{ \delta \in \Delta :~ x_N^* \notin \bigcap_{i=1}^m X_i(\delta) \Big \} \nonumber \\
&~~~~~\leq \sum_{i=1}^m \varepsilon_i (d_{i,N}) \Big \} \geq 1 - \sum_{i=1}^m \beta_i. \label{eq:prob_all_sub_new}
\end{align}
Unlike \eqref{eq:prob_all_cons_bar} and \eqref{eq:prob_all_cons}, \eqref{eq:prob_all_sub_new} is an a-posteriori statement due to the dependency of $\varepsilon_i (d_{i,N})$ on the extracted scenarios.
However, the sought a-priori result can be obtained by considering the worst-case value for $\sum_{i=1}^m \varepsilon_i (d_{i,N})$, with respect to the different combinations of $d_{i,N}$, $i=1,\ldots,m$, satisfying $\sum_{i=1}^m d_{i,N} \leq d$.
This can be achieved by means of the following maximization problem:
\begin{align}
\varepsilon = &\max_{\{d_i \in \mathbb{N}_+ \}_{i=1}^m} \sum_{i=1}^m \varepsilon_i(d_i) \label{eq:opt_eps1}\\
&\text{subject to } \sum_{i=1}^m d_i \leq d, \nonumber
\end{align}
Problem \eqref{eq:opt_eps1} is an integer optimization program. It can be solved numerically to obtain $\varepsilon$.
The optimal value $\varepsilon$ of the problem above depends on $\{N_i,\beta_i\}_{i=1}^m$ and $d$, but this dependency is suppressed to simplify notation. Notice the slight abuse of notation, since $\{d_i\}_{i=1}^m$ in \eqref{eq:opt_eps1} are integer decision variables and should not be related to $\{d_{i,N}\}_{i=1}^m$.
\sg{We have the following theorem which is the main achievement of this section.}

\begin{figure}[t]
\centering
\includegraphics[width = \columnwidth]{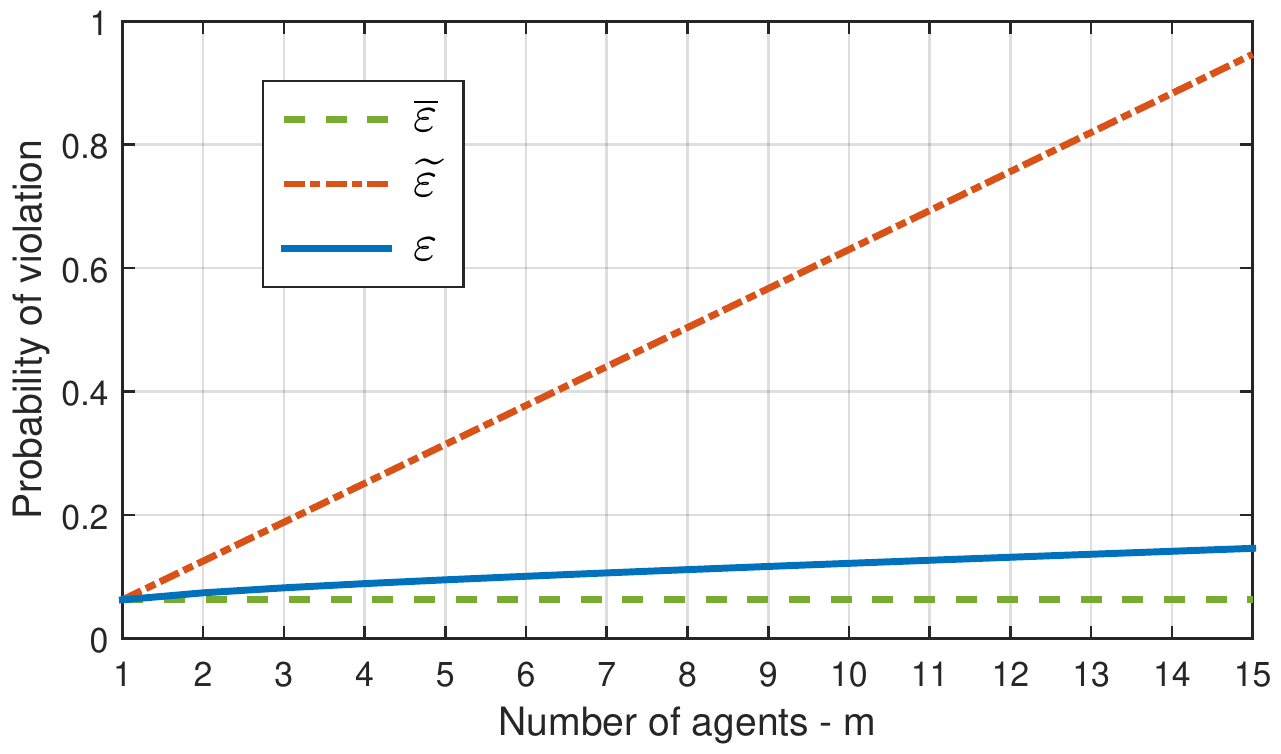}
\caption{Probability of constraint violation as a function of the number of agents, for the case where $d=50$, $\beta = 10^{-5}$, $N_i = \bar{N} = 4500$ and $\beta_i = \beta/m$, for all $i=1,\ldots,m$. The probability of violation $\bar{\varepsilon}$ (green dashed line) for the case of Section \ref{sec:SecVB} is independent of $m$, so it remains constant as the number of agents $m$ increases. For the case of Section \ref{sec:SecVB1}, $\widetilde{\varepsilon} \approx m \bar{\varepsilon}$ (red dotted-dashed line) for the considered set-up, so it grows approximately linearly with $m$. For the case of Section \ref{sec:SecVB2}, $\varepsilon$ (blue solid line) is moderately increasing with $m$, thus offering a less conservative result compared to the approach of Section \ref{sec:SecVB1}, while, in contrast to the approach of Section \ref{sec:SecVB}, it allows for distributed information about the scenarios.}\label{fig:guarantees}
\end{figure}

\begin{thm} \label{thm:prob_feas}
Fix $\beta \in (0,1)$ and choose $\beta_i$, $i=1,\ldots,m$, such that $\sum_{i=1}^m \beta_i = \beta$. Set $\varepsilon$ according to \eqref{eq:opt_eps1}. We then have that
\begin{align}
&\mathbb{P}^N \Big \{ S \in \Delta^N :~\mathbb{P} \Big \{ \delta \in \Delta :~ x_N^* \notin \bigcap_{i=1}^m X_i(\delta) \Big \} \leq \varepsilon \Big \} \geq 1-\beta. \label{eq:prob_all_new}
\end{align}
\end{thm}

\begin{proof}
Fix $\beta \in (0,1)$ and choose $\beta_i$, $i=1,\ldots,m$, such that $\sum_{i=1}^m \beta_i = \beta$.
Consider any set $S$ of scenarios and notice that $\sum_{i=1}^m d_{i,N} (S) \leq d$. This implies that $\{d_{i,N} (S)\}_{i=1}^m$ constitute a feasible solution of \eqref{eq:opt_eps1}.
Due to the fact that $\varepsilon$ is the optimal value of \eqref{eq:opt_eps1}, $\sum_{i=1}^m \varepsilon_i( d_{i,N} (S) ) \leq \varepsilon$ for any $S$, which together with \eqref{eq:prob_all_sub_new}, leads to \eqref{eq:prob_all_new} and hence concludes the proof.
\end{proof}

The result of Theorem \ref{thm:prob_feas} can be significantly less conservative compared to that of Proposition \ref{prob:prob_feas_sub}, since we explicitly account for the fact that $\sum_{i=1}^m d_{i,N} \leq d$ in the maximization problem in \eqref{eq:opt_eps1}. This can be also observed by means of the numerical example of Fig. \ref{fig:guarantees}, where we investigate how $\bar{\varepsilon}$, $\widetilde{\varepsilon}$ and $\varepsilon$ change as a function of the number of agents $m$. We consider a particular case where $d=50$, $\beta = 10^{-5}$, $N_i = \bar{N} = 4500$ and $\beta_i = \beta/m$, for all $i=1,\ldots,m$. For this set-up, where $\beta$ is split evenly among agents and all agents have the same number of scenarios, it turned out that the maximum value $\varepsilon$ in \eqref{eq:opt_eps1} is achieved for $d_i = d/m$, $i=1,\ldots,m$. As it can be seen from Fig. \ref{fig:guarantees}, $\bar{\varepsilon}$ (green dashed line) for the case of Section \ref{sec:SecVB} is independent of $m$, so it remains constant as the number of agents $m$ increases. For the case of Section \ref{sec:SecVB1}, \sg{$\widetilde{\varepsilon}$ (red dotted-dashed line) rows approximately linearly with $m$ (see also discussion at the end of Section \ref{sec:SecVB1})}. For the case of Section \ref{sec:SecVB2}, $\varepsilon$ (blue solid line) is moderately increasing with $m$, thus offering a less conservative result compared to the approach of Section \ref{sec:SecVB1}, while, in contrast to the approach of Section \ref{sec:SecVB}, it allows for distributed information about the uncertainty scenarios.

In certain cases (e.g., when the number of agents is high), $\varepsilon$ may still exceed one and hence the result of Theorem \ref{thm:prob_feas} becomes trivial (the same for Proposition \ref{prob:prob_feas_sub} in such cases). Similarly to the discussion at the end of Section \ref{sec:SecVB}, Theorem \ref{thm:prob_feas} can be reversed to compute the number of scenarios $N_i$ that need to be extracted by agent $i$, $i=1,\ldots,m$, for a given value of $\varepsilon, \beta \in (0,1)$. This can be achieved by numerically seeking for values of $N_i$, $i=1,\ldots,m$, that lead to a solution of \eqref{eq:opt_eps1} that attains the desired $\varepsilon$.

\section{Numerical example} \label{sec:exam}

We address a multi-agent regression problem subject to $L_1$-regularization, which is inspired by Example 1 of \cite{Campi_Care_2013}. Specifically, we consider $m$ functions $s_i(\delta)$, $i=1,\ldots,m$, which can, for instance, represent the effect of the same phenomenon at different locations of $m$ agents. The functions are unknown, and each agent $i$ has access to a (private) data set $\{(\delta_i^{(j)},s_i(\delta_i^{(j)}),\ j=1,\ldots,N_i\}$ of measurements of function  $s_i(\delta)$ only.

The agents seek to determine the magnitude of $d$ co-sinusoids at given frequencies, so that their superposition provides a central approximation of all the $s_i(\cdot)$, $i=1,\ldots,m$. To this end, letting $x = [x^{[1]},\ldots,x^{[d]},x^{[d+1]}] \in \mathbb{R}^{d+1}$, the following program is considered:
\begin{align}
\min_{x \in X \subset \mathbb{R}^{d+1}} & x^{[d+1]} + \lambda \|x\|_1 \label{eq:reg} \\
\text{subject to }& \Big |\sum_{\ell=1}^d x^{[\ell]} \cos(\ell \delta_i^{(j)}) - s_i(\delta_i^{(j)}) \Big| \leq x^{[d+1]}, \nonumber \\
&\text{for all } j=1,\ldots,N_i, \text{ for all } i=1,\ldots,m.\nonumber
\end{align}
In \eqref{eq:reg}, one minimizes  $x^{[d+1]}$, which is the worst-case approximation error over the data-points of all agents, plus a regularization term $\lambda \|x\|_1$, which induces sparsity in the solution. The set $X$ is  a hyper-rectangular with high enough edge length so that the solution remains the same as in the unconstrained case, and it is introduced to ensure compactness so that Algorithm \ref{alg:Alg1} can be applied (in fact this set could be different per agent, and does not need to be agreed upfront).
By setting $f_i(x) = (1/m) (x^{[d+1]} + \lambda \|x\|_1)$, $X_i(\delta) = \{x \in X:~\Big |\sum_{\ell=1}^d x^{[\ell]} \cos(\ell \delta) - s_i(\delta) \Big| \leq x^{[d+1]} \}$, and $S_i = \{\delta_i^{(1)},\ldots,\delta_i^{(N_i)}\}$, $i=1,\ldots,m$, it is seen that problem \eqref{eq:reg} is in the form of $\mathcal{P}_N$, and, moreover, it satisfies the assumptions of Corollary \ref{cor:optimality_N}. Hence, the distributed Algorithm \ref{alg:Alg1} can be employed to compute the optimal solution of \eqref{eq:reg}. 
Notice that $x_i$ in Algorithm \ref{alg:Alg1} corresponds to a copy of $x$ maintained by agent $i$ and should not be confused with $x^{[\ell]}$, which is the $\ell$-th component of $x$. Each objective function $f_i(x)$ is non-differentiable.
In our simulation, we considered $m=6$ agents on a ring of alternating communicating pairs (time-varying communication graph), and assigned at each step the same weight to both the local solution and that transmitted by the active neighbor. Moreover, we set $n=d+1=51$, $\lambda = 0.001$, and $N_i = N = 4500$ for all $i=1,\ldots,m$. All samples $\delta_i^{(j)}$, $i=1,\ldots,m$, $j=1,\ldots,N$, were independently drawn from a uniform distribution with support $[-\pi,\pi]$, while, mimicking \cite{Campi_Care_2013}, each $s_i(\delta_i^{(j)})$ was obtained by evaluating the sum of a certain number of randomly shifted co-sinusoids. Finally, Algorithm \ref{alg:Alg1} was initialized with the solutions satisfying the local constraints only and $c(k) = 0.05/(k+1)$.
Figure \ref{fig:regr} shows the data points for each agent (grey dots) and the functions $\sum_{\ell=1}^d x^{[\ell]}_i \sin(\ell \delta)$ corresponding to the agents' solutions returned by Algorithm \ref{alg:Alg1} (a) at the initialization and (b) after $150$ iterations. As it appears, in conformity to Corollary \ref{cor:optimality_N}, all local solutions converge to a unique solution. The fact that this solution is also optimal can be experimentally  inspected from Figure \ref{fig:iter}, where the objective values corresponding to the agent local solutions as iterations progress are displayed against the optimal objective value of problem \eqref{eq:reg} computed via a centralized algorithm for comparison purposes. The value to which $x^{[d+1]}_i$ converged was $0.88$.
\begin{figure}[t]
\centering
\includegraphics[width = .49\columnwidth]{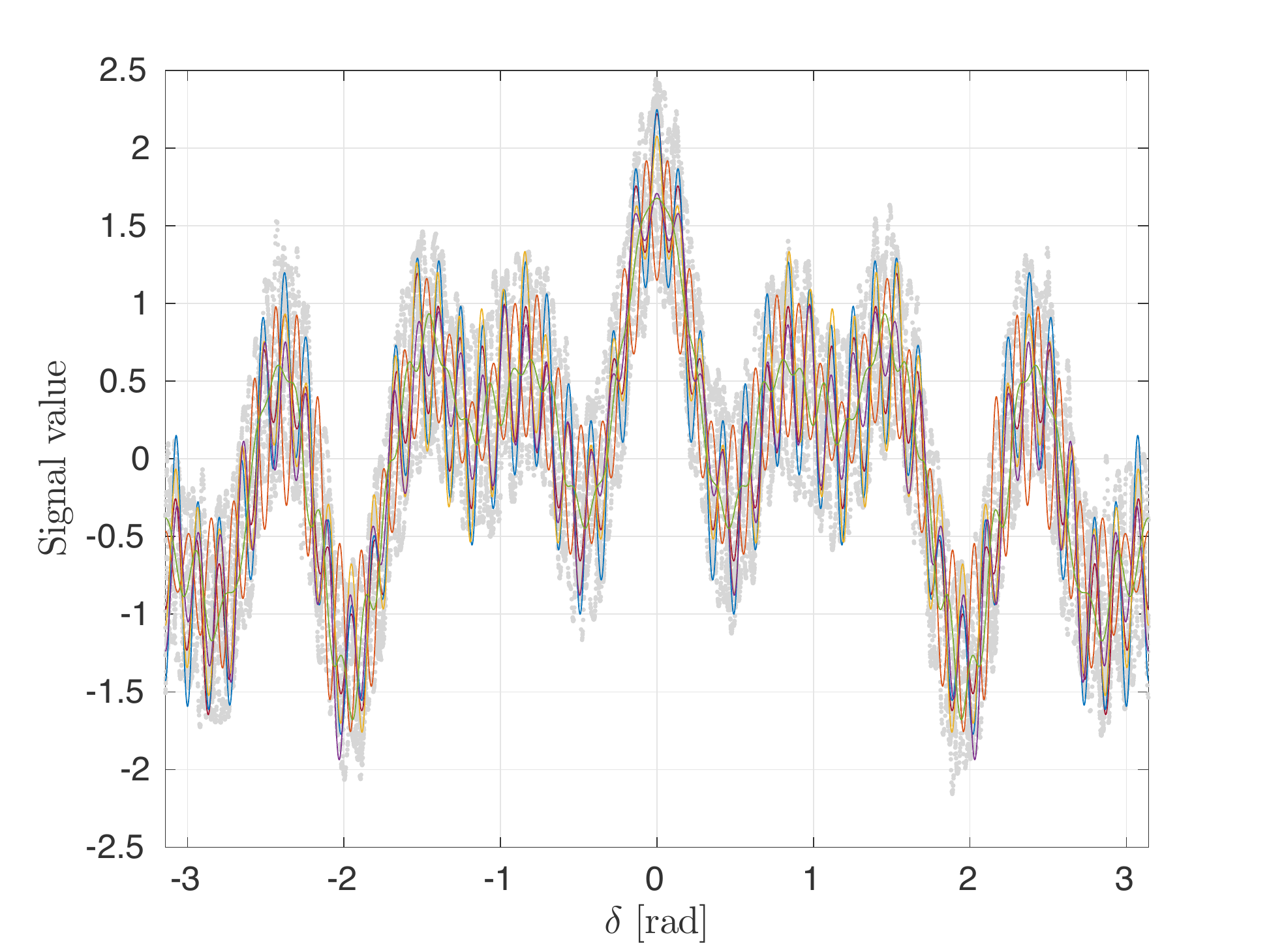}
\includegraphics[width = .49\columnwidth]{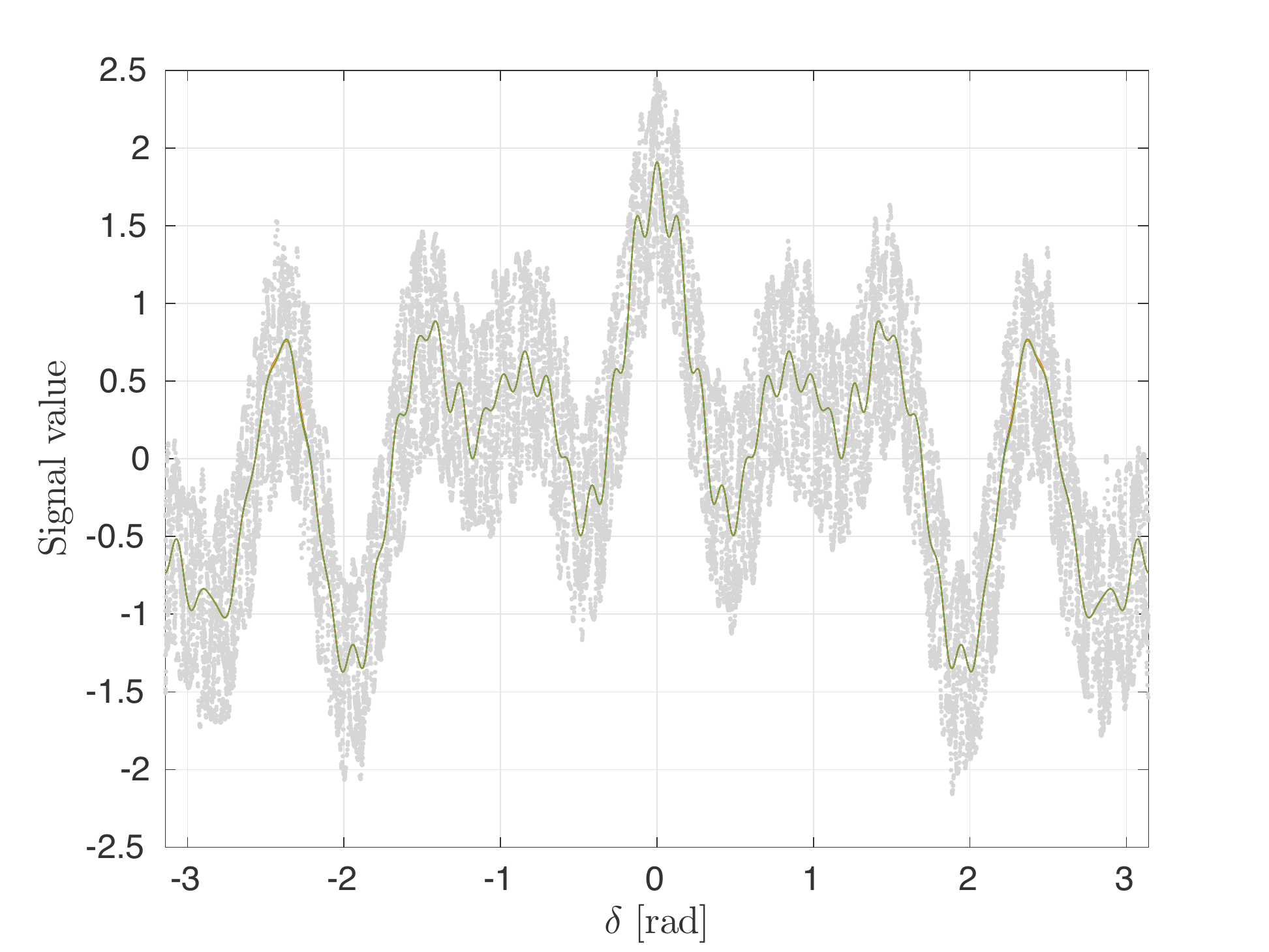}
\quad \quad (a) \hspace{3.7cm} (b)
\caption{Data points (grey crosses), and the functions (solid lines) corresponding to the local solutions returned by Algorithm \ref{alg:Alg1} (a) at the initialization and (b) after $150$ iterations.}\label{fig:regr}
\end{figure}
\begin{figure}[t]
\centering
\includegraphics[width = \columnwidth]{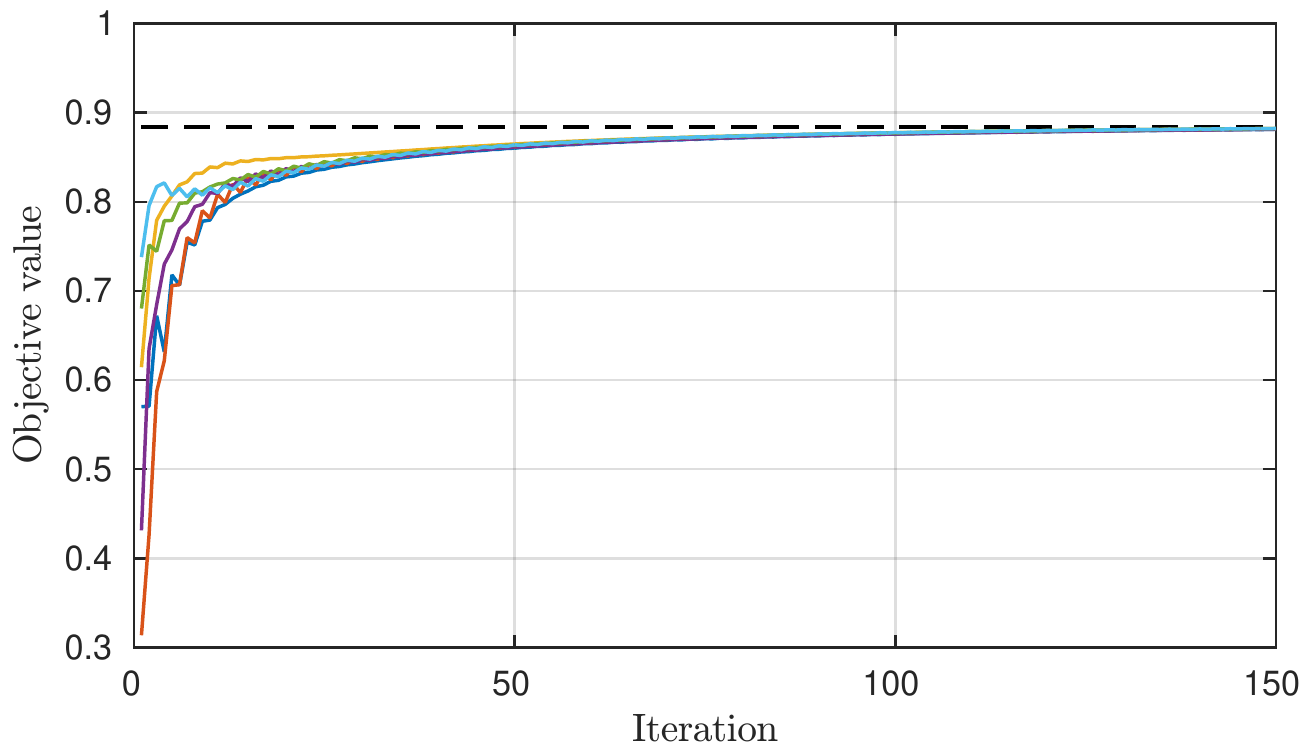}
\caption{Objective values corresponding to the agent local solutions as iterations progress (solid lines) vs. optimal value of problem \eqref{eq:reg} computed via a centralized algorithm (dashed line)}\label{fig:iter}
\end{figure}
In our simulations, scenarios were treated as private resources as each agent's scenarios are independent of the scenarios of other agents. Nonetheless, for a newly seen observation $\delta$, one may be interested in assessing the joint-constraint violation probability $\mathbb{P} \Big \{ \delta \in \Delta :~ x^*_N \notin X_i(\delta) \Big \}$, which in the present example corresponds to the probability of being apart from the obtained central function $\sum_{\ell=1}^d x^{[\ell]}_i \sin(\ell \delta)$ more than $0.88$ for at least one of the function $s_i(\delta)$, $i=1,\ldots,m$. Using $80000$ new scenarios (different from those used in the optimization process), this probability was empirically estimated as $0.01$. Using $\beta = 10^{-5}$ and $d=50$ (the bound on the dimension of the support set is $d = 50$ and not $d+1$, since we do not need to account for the epigraphic variable $x^{[d+1]}$, see \cite{Schildbach_etal_2013}), Proposition \ref{prob:prob_feas_sub} and Theorem \ref{thm:prob_feas} give $\widetilde{\varepsilon} = 0.37$ and $\varepsilon = 0.097$, respectively. As it can be seen, the novel bound of Theorem \ref{thm:prob_feas} provides a much tighter guaranteed upper bound for the probability of joint-constraint violation compared to $\widetilde{\varepsilon}$, while not requiring agents to have access to the same set of scenarios. Other runs of the example, with new observations extracted, always gave an estimate of the joint-constraint violation probability smaller than $0.09$, as it was expected given the high-confidence $1-10^{-5}$ with which the bound is guaranteed.

\section{Convergence analysis and proof of Theorem \ref{thm:optimality}} \label{sec:SecIII}
\subsection{Preparatory results} \label{sec:SecIIIAA}
We establish several relations between the difference of the agent estimates from certain average quantities.
At the end of this subsection we provide a summability result that is fundamental for the proof of Theorem \ref{thm:optimality} in subsection \ref{sec:SecIV}.

Let
\begin{align}
v(k) = \frac{1}{m} \sum_{i=1}^m x_i(k), \text{ for all } k \geq 0. \label{eq:v_k}
\end{align}
\sg{By using} Assumption \ref{ass:Convex}, the fact that the sets $X_i$, $i=1,\ldots,m$ are closed \sg{thanks to} Assumption \ref{ass:CompactLip}, and Assumption \ref{ass:SlaterPoint}, it is shown in Lemma 2 of \cite{Nedic_etal_2010} that
\begin{align}
\bar{v}(k) = \frac{\epsilon(k)}{\epsilon(k) + \rho} \bar{x} + \frac{\rho}{\epsilon(k) + \rho} v(k) &\in \bigcap_{i=1}^m X_i, \nonumber \\
&\text{ for all } k \geq 0, \label{eq:bar_v_k}
\end{align}
where $\epsilon(k) = \sum_{i=1}^m \dist(v(k),X_i)$, and $\bar{x} \in \mathbb{R}^n$, $\rho \in \mathbb{R}_+$ are as in Assumption \ref{ass:SlaterPoint}.
Note that unlike $x_i(k)$ and $v(k)$, which do not necessarily belong to $\bigcap_{i=1}^m X_i$, for $\bar{v}(k)$ this is always the case, thus providing a feasible solution of $\mathcal{P}$.

For each $i=1,\ldots,m$, denote by
\begin{align}
e_i(k+1) = x_i(k+1) - z_i(k), \text{ for all } k \geq 0, \label{eq:error}
\end{align}
the error between the values computed at steps 7 and 8 of Algorithm \ref{alg:Alg1}, i.e., the difference of the weighted average $z_i(k)$ computed by agent $i$ at time $k$ from its local update $x_i(k+1)$.

\subsubsection{Error relations} \label{sec:SecIIIA}
We provide some intermediate results that form the basis of the subsequent summability result.

\begin{lemma} \label{lemma:error_bar_v_k}
Consider Assumptions \ref{ass:Convex}, \ref{ass:CompactLip} and \ref{ass:SlaterPoint}. For all $k \geq 0$,
\begin{align}
\sum_{i=1}^m \|x_i(k) - \bar{v}(k)\| \leq \mu \sum_{i=1}^m \|x_i(k) - v(k)\|, \label{eq:error_bar_v_k}
\end{align}
where $\mu = (2/\rho) m D + 1$, with $D$ given below Assumption \ref{ass:CompactLip}.
\end{lemma}

From step 7 of Algorithm \ref{alg:Alg1} we have that for all $k \geq 0$, for all $i=1,\ldots,m$,
\begin{align}
x_i(k+1) &= \sum_{j=1}^m a_j^i(k) x_j(k) + x_i(k+1) - z_i(k) \nonumber \\
&=  \sum_{j=1}^m a_j^i(k) x_j(k) + e_i(k+1), \label{eq:dyn_sys}
\end{align}
where the last equality follows from \eqref{eq:error}.

Following \cite{Nedic_Ozdaglar_2009}, for each $k \geq 0$ consider a matrix $A(k) \in \mathbb{R}_+^{m \times m}$ such that $a_j^i(k)$ is the $j$-th element of its $i$-th column. For all $k, s$ with $k \geq s$, let $\Phi(k,s) = A(s) A(s+1) \ldots A(k-1) A(k)$, with $\Phi(k,k) = A(k)$ for all $k \geq 0$.
Denote by $\big [ \Phi(k,s) \big ]_j^i$ element $j$ of column $i$ of $\Phi(k,s)$. It is then shown in \cite{Nedic_Ozdaglar_2009} that, under Assumption \ref{ass:Weights}, $\Phi(k,s)$ is doubly stochastic.
Similarly to \cite{Nedic_Ozdaglar_2009}, by propagating \eqref{eq:dyn_sys} in time, it can be shown that for all $k > s$ (the inequality is strict for convenience of the subsequent derivations), for all $i=1,\ldots,m$,
\begin{align}
x_i(k&+1) = \sum_{j=1}^m \big [ \Phi(k,s) \big ]_j^i x_j(s) \nonumber \\
& + \sum_{r=s}^{k-1} \sum_{j=1}^m \big [ \Phi(k,r+1) \big ]_j^i e_j(r+1) + e_i(k+1). \label{eq:dyn_sys_x}
\end{align}
For all $k>s$, the last statement, together with \eqref{eq:v_k} and the fact that $\Phi(k,s)$ is a doubly stochastic matrix, leads to
\begin{align}
v(k&+1) = \frac{1}{m} \sum_{j=1}^m x_j(s) \nonumber \\
& + \frac{1}{m} \sum_{r=s}^{k-1} \sum_{j=1}^m e_j(r+1) + \frac{1}{m} \sum_{j=1}^m e_i(k+1). \label{eq:dyn_sys_v}
\end{align}

We then have the following lemma, which relates $\|x_i(k+1) - v(k+1)\|$ to $\|e_i(k+1)\|$, $i=1,\ldots,m$. Its proof follows from Lemma 8 in \cite{Nedic_etal_2010}.
\begin{lemma} \label{lemma:error_v_k}
Consider Assumptions \ref{ass:Network} and \ref{ass:Weights}. For all $k,s$ with $s \geq 0$, $k > s$, and for all $i=1,\ldots,m$,
\begin{align}
\|x_i(k+1) &- v(k+1)\| \leq \lambda q^{k-s} \sum_{j=1}^m \|x_j(s)\| \nonumber \\
& + \sum_{r=s}^{k-1} \lambda q^{k-r-1} \sum_{j=1}^m \|e_j(r+1)\| \nonumber \\
& + \|e_i(k+1)\| + \frac{1}{m} \sum_{j=1}^m \|e_j(k+1)\|, \label{eq:error_v_k}
\end{align}
where $\lambda = 2 \big ( 1 + \eta^{-(m-1)T} \big ) / \big ( 1 - \eta^{(m-1)T} \big ) \in \mathbb{R}_+$ and $q = \big ( 1 - \eta^{(m-1)T} \big )^{\frac{1}{(m-1)T}} \in (0,1)$.
\end{lemma}

\subsubsection{A summability relation} \label{sec:SecIIIB}
Let $N \in \mathbb{N}_+$ and consider the term
\begin{align}
2 \bar{L} \sum_{k=1}^N c(k) \sum_{i=1}^m \|x_i(k+1) - \bar{v}(k+1)\|, \label{eq:sum_error}
\end{align}
where $\bar{L} = \max_{i=1,\ldots,m} L_i$ with $L_i$ defined according to \eqref{eq:Lipschitz}.
We will show that \eqref{eq:sum_error} has an interesting relation with $\sum_{k=1}^N \sum_{i=1}^m \|e_i(k+1)\|^2$ and will come back to it often in the next section to establish certain summability results.

Consider Lemma \ref{lemma:error_bar_v_k} with $k+1$ in place of $k$ and Lemma \ref{lemma:error_v_k}, summing both sides of \eqref{eq:error_v_k} with respect to $i=1,\ldots,m$ and setting $s=0$. After some algebraic manipulations and index changes, we have that
\begin{align}
2 \bar{L} \sum_{k=1}^N &c(k) \sum_{i=1}^m \|x_i(k+1) - \bar{v}(k+1)\| \nonumber \\
&\leq 2 m \mu \lambda \bar{L} \sum_{k=1}^N  c(k) q^{k} \sum_{i=1}^m \|x_i(0)\| \nonumber \\
& + 2 m \mu \lambda \bar{L} \sum_{k=1}^N \sum_{r=0}^{k-1} c(k) q^{k-r-1} \sum_{i=1}^m \|e_i(r+1)\| \nonumber \\
& + 4 \mu \bar{L} \sum_{k=1}^N c(k) \sum_{i=1}^m \|e_i(k+1)\|. \label{eq:sum_error_all}
\end{align}

We then have the following lemma:
\begin{lemma} \label{lemma:sum_error}
Consider Assumptions \ref{ass:Convex}-\ref{ass:Weights}. Fix any $\alpha_1 \in (0,1)$, and consider \eqref{eq:v_k}-\eqref{eq:error}. We then have that for any $N \in \mathbb{N}_+$,
\begin{align}
2 &\bar{L} \sum_{k=1}^N c(k) \sum_{i=1}^m \|x_i(k+1) - \bar{v}(k+1)\| \nonumber \\
& < \alpha_1 \sum_{k=1}^N \sum_{i=1}^m \|e_i(k+1)\|^2 + \alpha_2 \sum_{k=1}^N c(k)^2 + \alpha_3, \label{eq:sum_e_bound}
\end{align}
where
\begin{align}
\alpha_2 &= \frac{2}{\alpha_1} m \mu^2 \bar{L}^2 \Big ( m^2 \lambda^2 \frac{1}{(1-q)^2} + 4 \Big ), \nonumber \\
\alpha_3 &= 2 m^3 \mu^2 \lambda^2 \bar{L}^2 c(0)^2 \frac{1}{\alpha_1(1-q)^2} \nonumber \\
& + 2 m^2 \mu \lambda \bar{L} D c(1) \frac{q}{1-q} + 2 \alpha_1 m D^2. \label{eq:constants_lm}
\end{align}
\end{lemma}

\subsection{Algorithm analysis} \label{sec:SecIV}
In this section we deal with the convergence properties of Algorithm \ref{alg:Alg1}, and provide the proof of Theorem \ref{thm:optimality}.

\subsubsection{Error convergence} \label{sec:SecIVA}
We prove convergence properties for the error in \eqref{eq:error}, which are instrumental to the proof of Theorem \ref{thm:optimality}.
We use the following result, which is proven in Lemma 4.1 in \cite{Bertsekas_Tsitsiklis_1997} (p. 257) for the case where the constraint sets are polyhedral. As mentioned in p. 662 of the same reference, the assertion of the lemma remains valid also in the case of general convex constraint sets. For the latter we refer the reader to \cite{Rockafellar_1970}, and to \cite{nedic_scaglione_2014} (Lemma 9) for a recent use of the lemma in case of convex constraint sets.

\begin{lemma} [Lemma 4.1 in \cite{Bertsekas_Tsitsiklis_1997} (p. 257)] \label{lemma:book}
If $y^* = \arg \min_{y \in Y} J_1(y) + J_2(y)$ (assuming uniqueness of the minimizer), where $Y \subseteq \mathbb{R}^n$ is a closed, convex set, $J_1(\cdot), J_2(\cdot):~ \mathbb{R}^n \rightarrow \mathbb{R}$ are convex functions and $J_2(\cdot)$ is continuously differentiable, then $y^* = \arg \min_{y \in Y} J_1(y) + \nabla J_2(y^*)^\top y$, where $\nabla J_2(y^*)$ is the gradient of $J_2(y)$ with respect to $y$, evaluated at $y^*$.
\end{lemma}

Consider step 8 of Algorithm \ref{ass:Convex}. \sg{Thanks to} Assumptions \ref{ass:Convex}, the fact that the sets $X_i$, $i=1,\ldots,m$ are closed \sg{by} Assumption \ref{ass:CompactLip},  Assumption \ref{ass:SlaterPoint}, and \sg{the fact that} $(1/(2c(k))) \|z_i(k) - x_i\|^2$ is continuously differentiable with respect to $x_i$, Lemma \ref{lemma:book} \sg{can be applied to the} problem with $x_i, X_i$ in place of $y, Y$, respectively, $f_i(x_i)$ in place of $J_1(y)$ and $(1/(2c(k))) \|z_i(k) - x_i\|^2$ in place of $J_2(y)$. We have that
\begin{align}
x_i(k+1) = \arg \min_{x_i \in X_i} &f_i(x_i) \nonumber \\
&- \frac{1}{c(k)} (z_i(k) - x_i(k+1))^\top x_i, \label{eq:sep_obj}
\end{align}
where in the second term of \eqref{eq:sep_obj}, $- (1/c(k)) (z_i(k) - x_i(k+1))$, is the gradient of $(1/(2c(k))) \|z_i(k) - x_i\|^2$ with respect to $x_i$, evaluated at $x_i(k+1)$. We then have the following lemma, which provides a useful relation between the consecutive algorithm iterates $x_i(k+1)$ and $x_i(k)$, and we will be using it extensively in the subsequent results. The subsequent proof strongly depends on the use of Lemma \ref{lemma:book}, and deviates from the proofs of the basic iterate relations in \cite{Nedic_etal_2010} (Lemma 6) and \cite{Lee_Nedic_2013} (Lemma 5); it is motivated by the proof of the alternating direction method of multipliers (Proposition 4.2 in \cite{Bertsekas_Tsitsiklis_1997}, Appendix A of \cite{Boyd_etal_2010}), and relies on our proximal minimization perspective.

\begin{lemma} \label{lemma:relation}
Consider Assumptions \ref{ass:Convex}, \ref{ass:CompactLip}, \ref{ass:SlaterPoint} and \ref{ass:Weights}. We then have that for any $k \in \mathbb{N}_+$, for any $x^* \in X^*$,
\begin{align}
2c(k)&\sum_{i=1}^m f_i(\bar{v}(k+1)) + \sum_{i=1}^m \|e_i(k+1)\|^2 \nonumber \\
&~~~~~~~~~~+ \sum_{i=1}^m \|x_i(k+1) - x^*\|^2 \nonumber \\
& \leq 2c(k) \sum_{i=1}^m f_i(x^*) + \sum_{i=1}^m \|x_i(k)-x^*\|^2 \nonumber \\
&~~~~~~~~~~+ 2\bar{L} c(k)\sum_{i=1}^m \|x_i(k+1) - \bar{v}(k+1)\|, \label{eq:sep_obj4}
\end{align}
where $e_i(k+1)$ is given as in \eqref{eq:error}.
\end{lemma}

\begin{proof}
\sg{Thanks to Lemma \ref{lemma:book}, \eqref{eq:sep_obj} holds true}. Since $x_i(k+1) \in X_i$ is the minimizer of the optimization problem that appears in the right-hand side of \eqref{eq:sep_obj}, we have that
\begin{align}
f_i(x_i&(k+1)) - \frac{1}{c(k)} (z_i(k) - x_i(k+1))^\top x_i(k+1) \nonumber \\
& \leq f_i(x) - \frac{1}{c(k)} (z_i(k) - x_i(k+1))^\top x, \nonumber \\
&~~~~~~~~~~~~~~~~~~~~~~~~~~~~~\text{ for all } x \in X_i.\label{eq:sep_obj1}
\end{align}
Since the last statement holds for any $x \in X_i$, it will also hold for any minimizer $x^* \in X^* \subseteq \bigcap_{i=1}^m X_i$ of problem $\mathcal{P}$ in \eqref{eq:P_con}.

We have that for any $x^* \in X^*$,
\begin{align}
- (z_i(k) &- x_i(k+1))^\top (x_i(k+1)-x^*) \nonumber \\ &= \frac{1}{2} \|x_i(k+1) - z_i(k)\|^2
+ \frac{1}{2} \|x_i(k+1) - x^*\|^2 \nonumber \\
&- \frac{1}{2} \|z_i(k)-x^*\|^2. \label{eq:square_id}
\end{align}

By \eqref{eq:sep_obj1}, \eqref{eq:square_id}, we have that for any $x^* \in X^*$,
\begin{align}
f_i(&x_i(k+1)) + \frac{1}{2c(k)} \|x_i(k+1) - z_i(k)\|^2 \nonumber \\
&~~~~~~~~~~~~+ \frac{1}{2c(k)} \|x_i(k+1) - x^*\|^2 \nonumber \\
& \leq f_i(x^*) + \frac{1}{2c(k)} \|z_i(k)-x^*\|^2 \nonumber \\
&\leq f_i(x^*) + \frac{1}{2c(k)} \sum_{j=1}^m a_j^i(k) \|x_j(k)-x^*\|^2,\label{eq:sep_obj2}
\end{align}
where the last inequality follows by the definition of $z_i(k)$ (see step 7 of Algorithm \ref{alg:Alg1}), the fact that, under Assumption \ref{ass:Weights}, $\|\sum_{j=1}^m a_j^i(k) x_j(k) - x^* \|^2 = \|\sum_{j=1}^m a_j^i(k) \big (x_j(k) - x^*\big) \|^2$ and the convexity of $\|\cdot\|^2$.

Multiply both sides of \eqref{eq:sep_obj2} by $2c(k)$, sum with respect to $i=1,\ldots,m$, and notice that for any $k \geq 0$, under the double stochasticity condition of Assumption \ref{ass:Weights}, $\sum_{i=1}^m a_j^i(k) = 1$. We then have that
\begin{align}
2c(k)&\sum_{i=1}^m f_i(x_i(k+1)) + \sum_{i=1}^m \|x_i(k+1) - z_i(k)\|^2 \nonumber \\
&~~~~~~~~~~~~~~~~~~~~~+ \sum_{i=1}^m \|x_i(k+1) - x^*\|^2 \nonumber \\
& \leq 2c(k) \sum_{i=1}^m f_i(x^*) + \sum_{i=1}^m \|x_i(k)-x^*\|^2. \label{eq:sep_obj3}
\end{align}

Consider Assumption \ref{ass:SlaterPoint}, and let $\bar{v}(k)$ be as in \eqref{eq:bar_v_k}.
Under Assumptions \ref{ass:Convex} and \ref{ass:CompactLip}, by \eqref{eq:Lipschitz} we have that $f_i(x_i(k+1)) \geq f_i(\bar{v}(k+1))-\bar{L} \|x_i(k+1) - \bar{v} (k+1)\|$, where $\bar{L} = \max_{i=1,\ldots,m} L_i$. Recall also that $\|x_i(k+1) - z_i(k)\| = \|e_i(k+1)\|$ by \eqref{eq:error}. Therefore, for any $x^* \in X^*$, the last statements together with \eqref{eq:sep_obj3}, lead to \eqref{eq:sep_obj4} and hence conclude the proof.
\end{proof}

We then have the following proposition.

\begin{proposition} \label{prop:conv_error}
Consider Assumptions \ref{ass:Convex}-\ref{ass:Weights} and Algorithm \ref{alg:Alg1}. We have that
\begin{enumerate}
\item $\sum_{k=1}^{\infty} \sum_{i=1}^m \|e_i(k)\|^2 < \infty$,
\item $\lim_{k \rightarrow \infty} \|e_i(k)\| = 0$, for all $i=1,\ldots,m$,
\end{enumerate}
where $e_i(k)$ is given as in \eqref{eq:error}.
\end{proposition}

\begin{proof}
By Lemma \ref{lemma:relation}, \eqref{eq:sep_obj4} holds. Fix any $\alpha_1 \in (0,1)$. Under Assumption \ref{ass:CompactLip}-\ref{ass:Network}, let $\alpha_2, \alpha_3$ as in \eqref{eq:constants_lm}, and consider \eqref{eq:sum_e_bound}. Sum then \eqref{eq:sep_obj4} with respect to $k = 1,\ldots,N$ for an arbitrary $N \in \mathbb{N}_+$, and upper-bound the resulting last term in the right-hand side of \eqref{eq:sep_obj4} using \eqref{eq:sum_e_bound}. We then have that, for all $x^* \in X^*$,
\begin{align}
2\sum_{k=1}^N &c(k)\sum_{i=1}^m \Big ( f_i(\bar{v}(k+1)) - f_i(x^*)\Big ) \nonumber \\
&~~~~~~~~~~ + (1-\alpha_1) \sum_{k=1}^N \sum_{i=1}^m \|e_i(k+1)\|^2  \nonumber \\
& \leq \sum_{i=1}^m \|x_i(1)-x^*\|^2 - \sum_{i=1}^m \|x_i(N+1)-x^*\|^2 \nonumber \\
&~~~~~~~~~~+ \alpha_2 \sum_{k=1}^N c(k)^2 + \alpha_3. \label{eq:sep_obj5}
\end{align}

Since $\bar{v}(k+1) \in \bigcap_{i=1}^m X_i$ for all $k \geq 0$, and $x^*$ is a minimizer of $\mathcal{P}$, $2\sum_{k=1}^N c(k)\sum_{i=1}^m \Big ( f_i(\bar{v}(k+1)) - f_i(x^*)\Big ) \geq 0$. Moreover, $\sum_{i=1}^m \|x_i(N+1)-x^*\|^2 \geq 0$, hence these two terms can be dropped from \eqref{eq:sep_obj5}. Therefore, by \eqref{eq:sep_obj5},
\begin{align}
(1-\alpha_1) &\sum_{k=1}^N \sum_{i=1}^m \|e_i(k+1)\|^2  \nonumber \\
& \leq \sum_{i=1}^m \|x_i(1)-x^*\|^2 + \alpha_2 \sum_{k=1}^N c(k)^2 + \alpha_3. \label{eq:sep_obj6}
\end{align}

Let now $N \rightarrow \infty$. Since $\alpha_1 \in (0,1)$ and $\sum_{i=1}^m \|x_i(1)-x^*\|^2 + \alpha_2 \sum_{k=1}^{\infty} c(k)^2 + \alpha_3 < \infty$, by Assumptions \ref{ass:CompactLip} and \ref{ass:ConvCoef}, \eqref{eq:sep_obj6} implies that $\sum_{k=1}^{\infty} \sum_{i=1}^m \|e_i(k+1)\|^2 < \infty$, and hence also $\sum_{k=1}^{\infty} \sum_{i=1}^m \|e_i(k)\|^2 < \infty$, thus establishing the first part of the proposition.
The second part directly follows from the fact that $\sum_{k=1}^{\infty} \sum_{i=1}^m \|e_i(k)\|^2 < \infty$ and $\|e_i(k)\|$ is a non-negative quantity, thus concluding the proof.
\end{proof}

\subsubsection{Average tracking} \label{sec:SecIVC}
We show that the agents' estimates $x_i(k)$, $i=1,\ldots,m$, track their arithmetic average $v(k)$, in the sense that
$\lim_{k \rightarrow \infty} \|x_i(k) - v(k)\| = 0$ for all $i=1,\ldots,m$. This is summarized in the following proposition. The proof follows from Lemma 8 in \cite{Nedic_etal_2010}, however, we include it also here for completeness.
\begin{proposition} \label{prop:conv_consensus}
Consider Assumptions \ref{ass:Convex}-\ref{ass:Weights} and Algorithm \ref{alg:Alg1}. We have that
\begin{align}
\lim_{k \rightarrow \infty} \|x_i(k) - v(k)\| = 0, \text{ for all } i=1,\ldots,m, \label{eq:consensus}
\end{align}
where $v(k)$ is given by \eqref{eq:v_k}.
\end{proposition}

\begin{proof}
Under Assumptions \ref{ass:Convex}-\ref{ass:Network}, by the second part of Proposition \ref{prop:conv_error} we have that $\lim_{k \rightarrow \infty} \|e_i(k)\| = 0$, for all $i=1,\ldots,m$. Then, for any $\epsilon > 0$ we can choose $s > 0$ such that $\|e_i(k)\| \leq \epsilon$ for all $k > s$, for all $i=1,\ldots,m$.

By \eqref{eq:error_v_k} of Lemma \ref{lemma:error_v_k}, we then have that for all $i=1,\ldots,m$,
\begin{align}
\|x_i(k+1) &- v(k+1)\| \leq \lambda q^{k-s} \sum_{j=1}^m \|x_j(s)\| \nonumber \\
& ~~~~~~~~~~~~~~~+ m \lambda \epsilon \sum_{r=s}^{k-1} q^{k-r-1} + 2 \epsilon \nonumber \\
& = \lambda q^{k-s} \sum_{j=1}^m \|x_j(s)\| + m \lambda \epsilon \sum_{t=0}^{k-s-1} q^{t} + 2 \epsilon \nonumber \\
& < \lambda q^{k-s} \sum_{j=1}^m \|x_j(s)\| + m \lambda \epsilon \sum_{t=0}^{\infty} q^{t} + 2 \epsilon \nonumber \\
& \leq m \lambda D q^{k-s} + m \lambda \frac{1}{1-q} \epsilon + 2\epsilon, \label{eq:error_v_k_1}
\end{align}
where the equality is due to a change of the summation limits, and the last inequality is due to Assumption \ref{ass:CompactLip} and the fact that $q \in (0,1)$.

Taking limit superior in both sides of \eqref{eq:error_v_k_1} as $k \rightarrow \infty$,
\begin{align}
\lim\sup_{k \rightarrow \infty}\|x_i(k+1) - v(k+1)\| \leq m \lambda \frac{1}{1-q} \epsilon + 2\epsilon. \label{eq:error_v_k_2}
\end{align}
Note that taking the limit superior as $k \rightarrow \infty$ is well defined, since $\epsilon$ is assumed to be fixed, and hence also $s$. Notice also that the resulting quantity in the right-hand side of \eqref{eq:error_v_k_2} no longer depends on $s$.
Since $\epsilon > 0$ is arbitrary, \sg{relation} \eqref{eq:error_v_k_2} implies that $\lim_{k \rightarrow \infty} \|x_i(k+1) - v(k+1)\| = 0$, and hence $\lim_{k \rightarrow \infty} \|x_i(k) - v(k)\| = 0$, for all $i=1,\ldots,m$, thus concluding the proof.
\end{proof}

\subsubsection{Convergence and optimality} \label{sec:SecIVB}
In this subsection we will provide a proof of Theorem \ref{thm:optimality}.
To achieve this, we will first show an intermediate convergence result. Notice that by the first part of Proposition \ref{prop:conv_error} (under Assumptions \ref{ass:Convex}-\ref{ass:Weights}), $\sum_{k=1}^{\infty} \sum_{i=1}^m \|e_i(k)\|^2 < \infty$.
Letting then $N \rightarrow \infty$ in \eqref{eq:sum_e_bound} leads to the following summability result, which states that
\begin{align}
2 \bar{L} \sum_{k=1}^{\infty} c(k) \sum_{i=1}^m \|x_i(k+1) - \bar{v}(k+1)\| < \infty. \label{eq:sum_result}
\end{align}
The last statement enables us to show the following convergence result.
\begin{thm} \label{thm:alg_conv}
Consider Assumptions \ref{ass:Convex}-\ref{ass:Weights} and Algorithm \ref{alg:Alg1}. We have that, for any minimizer $x^* \in X^*$, the sequence $\big \{ \|x_i(k) - x^*\| \big \}_{k \geq 0}$ is convergent for all $i=1,\ldots,m$.
\end{thm}

\begin{proof}
By Lemma \ref{lemma:relation}, \eqref{eq:sep_obj4} holds. Summing then \eqref{eq:sep_obj4} with respect to $k = M,\ldots,N$ for arbitrary $M, N \in \mathbb{N}_+$, we have that, for all $x^* \in X^*$,
\begin{align}
2&\sum_{k=M}^N c(k)\sum_{i=1}^m \Big ( f_i(\bar{v}(k+1)) - f_i(x^*)\Big ) \nonumber \\
&~~~~~~~+ \sum_{k=M}^N \sum_{i=1}^m \|e_i(k+1)\|^2  + \sum_{i=1}^m \|x_i(N+1)-x^*\|^2  \nonumber \\
& \leq \sum_{i=1}^m \|x_i(M)-x^*\|^2 \nonumber \\
&~~~~~~~+ 2 \bar{L} \sum_{k=M}^{N} c(k) \sum_{i=1}^m \|x_i(k+1) - \bar{v}(k+1)\|. \label{eq:sep_obj6n}
\end{align}

As in the proof of Proposition \ref{prop:conv_error}, notice that since $\bar{v}(k+1) \in \bigcap_{i=1}^m X_i$ for all $k \geq 0$, and $x^*$ is a minimizer of $\mathcal{P}$, $2\sum_{k=M}^N c(k)\sum_{i=1}^m \Big ( f_i(\bar{v}(k+1)) - f_i(x^*)\Big ) \geq 0$. Moreover, $\sum_{k=M}^N \sum_{i=1}^m \|e_i(k+1)\|^2 \geq 0$, hence these two terms can be dropped from the left-hand side of \eqref{eq:sep_obj6n}. Therefore, by \eqref{eq:sep_obj6n} we have that
\begin{align}
\sum_{i=1}^m \|&x_i(N+1)-x^*\|^2 \leq \sum_{i=1}^m \|x_i(M)-x^*\|^2 \nonumber \\
&+ 2 \bar{L} \sum_{k=M}^{N} c(k) \sum_{i=1}^m \|x_i(k+1) - \bar{v}(k+1)\|. \label{eq:sep_obj7}
\end{align}

Notice that, under Assumptions \ref{ass:Convex}-\ref{ass:Network}, the summability statement of \eqref{eq:sum_result} holds. Taking then in \eqref{eq:sep_obj7} the limit superior as $N \rightarrow \infty$ and the limit inferior as $M \rightarrow \infty$, we have that
\begin{align}
\lim \sup_{N \rightarrow \infty} \sum_{i=1}^m \|&x_i(N+1)-x^*\|^2 \nonumber \\
&\leq \lim \inf_{M \rightarrow \infty} \sum_{i=1}^m \|x_i(M)-x^*\|^2 . \label{eq:sep_obj8}
\end{align}
The last statement, together with the fact that the sequence $\big \{ \sum_{i=1}^m \|x_i(k) - x^*\| \big \}_{k \geq 0}$ is bounded due to Assumption \ref{ass:CompactLip}, implies that $\big \{ \sum_{i=1}^m \|x_i(k) - x^*\| \big \}_{k \geq 0}$ converges for all $x^* \in X^*$.

Consider now $v(k) = \frac{1}{m} \sum_{i=1}^m x_i(k)$. We have that
\begin{align}
\| v(k) - x^* \| = \| \frac{1}{m} \sum_{i=1}^m x_i(k) - x^* \| \leq \frac{1}{m} \sum_{i=1}^m \| x_i(k) - x^* \|. \label{eq:sep_obj9}
\end{align}
Moreover, \sg{from}
\begin{align}
\| x_i(k) - x^* \| \leq \|v(k) - x^* \| + \|x_i(k) - v(k) \|, \label{eq:sep_obj10}
\end{align}
\sg{it also holds that} $\|v(k) - x^* \| \geq \frac{1}{m} \sum_{i=1}^m \| x_i(k) - x^* \| - \frac{1}{m} \sum_{i=1}^m \|x_i(k) - v(k) \|$, \sg{which, together with} \eqref{eq:sep_obj9}, and since, for all $i=1,\ldots,m$, $\lim_{k \rightarrow \infty} \|x_i(k) - v(k)\| = 0$ by Proposition \ref{prop:conv_consensus}, and $\big \{ \sum_{i=1}^m \|x_i(k) - x^*\| \big \}_{k \geq 0}$ is convergent for all $x^* \in X^*$, \sg{gives} ${\| v(k) - x^* \|}_{k \geq 0}$ is also convergent for any $x^* \in X^*$. \sg{From \eqref{eq:sep_obj10} and}
\begin{align*}
\|v(k) - x^* \| - \|x_i(k) - v(k)\| \leq \| x_i(k) - x^* \|,
\end{align*}
the convergence of $\| v(k) - x^* \|$, along with $\lim_{k \rightarrow \infty} \|x_i(k) - v(k)\| = 0$, for all $i=1,\ldots,m$, gives the statement of the theorem.
\end{proof}

We are now in a position to prove Theorem \ref{thm:optimality} of Section \ref{sec:SecIIC}, showing that there exists some minimizer $x^* \in X^*$ of $\mathcal{P}$, such that $\lim_{k \rightarrow \infty} \|x_i(k) - x^*\| = 0$, for all $i=1,\ldots,m$, i.e., all agents reach consensus to a common minimizer of $\mathcal{P}$.

\begin{proofof}{Theorem \ref{thm:optimality}}
By Lemma \ref{lemma:relation}, \eqref{eq:sep_obj4} holds. Fix any $\alpha_1 \in (0,1)$. Under Assumptions \ref{ass:CompactLip}-\ref{ass:Network}, let $\alpha_2, \alpha_3$ as in \eqref{eq:constants_lm}, and consider \eqref{eq:sum_e_bound}. As in the proof of Proposition \ref{prop:conv_error}, sum \eqref{eq:sep_obj4} with respect to $k = 1,\ldots,N$ for an arbitrary $N \in \mathbb{N}_+$, and upper-bound the resulting last term in the right-hand side of \eqref{eq:sep_obj4} using \eqref{eq:sum_e_bound}. We then have that, for all $x^* \in X^*$, \eqref{eq:sep_obj5} holds.

Since $\sum_{k=1}^N \sum_{i=1}^m \|e_i(k+1)\|^2 \geq 0$ and $\sum_{i=1}^m \|x_i(N+1)-x^*\|^2 \geq 0$, we can drop the two terms in the left-hand side of \eqref{eq:sep_obj5}. Therefore, by \eqref{eq:sep_obj5}, we have that
\begin{align}
2\sum_{k=1}^N &c(k)\sum_{i=1}^m \Big ( f_i(\bar{v}(k+1)) - f_i(x^*)\Big ) \nonumber \\
& \leq \sum_{i=1}^m \|x_i(1)-x^*\|^2 + \alpha_2 \sum_{k=1}^N c(k)^2 + \alpha_3. \label{eq:sep_obj8}
\end{align}

Let now $N \rightarrow \infty$. Notice that, by Assumptions \ref{ass:CompactLip} and \ref{ass:ConvCoef}, $\sum_{i=1}^m \|x_i(1)-x^*\|^2 + \alpha_2 \sum_{k=1}^{\infty} c(k)^2 + \alpha_3 < \infty$. Therefore, $2\sum_{k=1}^{\infty} c(k)\sum_{i=1}^m \Big ( f_i(\bar{v}(k+1)) - f_i(x^*)\Big ) < \infty$; however, $\sum_{k=0}^{\infty} c(k) = \infty$, by Assumption \ref{ass:ConvCoef}. Hence,
\begin{align}
\lim \inf_{k \rightarrow \infty} \sum_{i=1}^m \Big ( f_i(\bar{v}(k+1)) - f_i(x^*)\Big ) = 0. \label{eq:opt_1}
\end{align}
Due to the continuity of $f_i(\cdot)$, $i=1,\ldots,m$, under the convexity requirement of Assumption \ref{ass:Convex}, \eqref{eq:opt_1} implies that there exists some $\bar{x}^* \in X^*$ such that
\begin{align}
\lim \inf_{k \rightarrow \infty} \|\bar{v}(k) - \bar{x}^*\| = 0. \label{eq:opt_2}
\end{align}
In other words, $\big \{\| \bar{v}(k) - \bar{x}^*\|\big \}_{k \geq 0}$ converges to $0$ across a subsequence.

By Proposition \ref{prop:conv_consensus} and Lemma \ref{lemma:error_bar_v_k} we have that $\lim_{k \to \infty} \|x_i(k) - \bar{v}(k)\| = 0$, for all $i=1,\ldots,m$.
Therefore, and since $\|x_i(k) - \bar{x}^*\| \leq \|\bar{v}(k) - \bar{x}^*\| + \|x_i(k) - \bar{v}(k)\|$, by \eqref{eq:opt_2} we have that, for all $i=1,\ldots,m$,
\begin{align}
\lim \inf_{k \rightarrow \infty} \|x_i(k) - \bar{x}^*\| = 0. \label{eq:opt_3}
\end{align}

On the other hand, it was shown in Theorem \ref{thm:alg_conv} that, for all $i=1,\ldots,m$, $\big \{ \|x_i(k) - x^*\| \big \}_{k \geq 0}$ converges for all $x^* \in X^*$, and hence also for $\bar{x}^*$. Hence, it must be $\lim_{k \to \infty} \|x_i(k) - \bar{x}^*\| = 0$, for all $i=1,\ldots,m$, which concludes the proof.
\end{proofof}

Note that a direct byproduct of Proposition \ref{prop:conv_consensus}, Theorem \ref{thm:optimality} and Lemma \ref{lemma:error_bar_v_k}, is that, there exists $x^* \in X^*$, such that $\lim_{k \rightarrow \infty} \|x_i(k) - x^*\| = \lim_{k \rightarrow \infty} \|v(k) - x^*\| = \lim_{k \rightarrow \infty} \|\bar{v}(k) - x^*\| = 0$, for all $i=1,\ldots,m$.

\section{Conclusion} \label{sec:secVI}
In this paper a unifying framework for distributed convex optimization over time-varying networks, in the presence of constraints and uncertainty is provided. We constructed an iterative, proximal minimization based algorithm, and analyzed its convergence and optimality properties.
To deal with the case where the agents' constraint sets are affected by a possibly common uncertainty vector, a scenario-based methodology was adopted, allowing agents to use a different set of uncertainty scenarios.

Current work concentrates on three main directions: 1) Investigating the convergence rate properties of the developed algorithm, and the potential of an asynchronous implementation. 2) Developing rolling horizon implementations, extending the work of \cite{Tsitsiklis_Athans_1984} to the case where constraints are also present. 3) Analyzing the quality of the scenario-based solutions, providing confidence intervals connecting the optimal values of $\mathcal{P}_{\bar{N}}$, $\mathcal{P}_N$ with the one of $\mathcal{P}_{\delta}$ by exploiting the results of \cite{Kanamori_Takeda_2012,Mohajerin_etal_2015}. 4) From an application point of view, the main focus is on applying the proposed algorithm to the problem of energy efficient control of a building network \cite{Ioli_etal_2015}.

\section*{Appendix} \label{sec:secApp}

\begin{proofof}{Lemma \ref{lemma:error_bar_v_k}}
We have that for all $k\geq 0$
\begin{align}
\|x_i(k) &- \bar{v}(k)\| \nonumber \\
& = \|\frac{\epsilon(k) + \rho}{\epsilon(k) + \rho} x_i(k) - \frac{\epsilon(k)}{\epsilon(k) + \rho} \bar{x} - \frac{\rho}{\epsilon(k) + \rho} v(k)\| \nonumber \\
& \leq \frac{1}{\epsilon(k) + \rho} \big ( \epsilon(k) \|x_i(k) - \bar{x}\| + \rho \|x_i(k) - v(k)\| \big ) \nonumber \\
& \leq \frac{1}{\rho} \big ( \epsilon(k) \|x_i(k) - \bar{x}\| + \rho \|x_i(k) - v(k)\| \big ), \label{eq:bar_v_k_1}
\end{align}
where the last inequality is due to the fact that $\epsilon(k) \geq 0$.

By the definition of $\dist(\cdot,\cdot)$, and since $x_i(k) \in X_i$ for all $i=1,\ldots,m$, we have that for all $k\geq 0$
\begin{align}
\epsilon(k) &= \sum_{i=1}^m \dist (v(k),X_i) \leq \sum_{i=1}^m \|x_i(k) - v(k)\|. \label{eq:bar_v_k_2}
\end{align}
By \eqref{eq:bar_v_k_1}, \eqref{eq:bar_v_k_2} we have that
\begin{align}
\|x_i(k) - \bar{v}(k)\| \leq \frac{1}{\rho} \Big ( \sum_{i=1}^m \|x_i(k) &- v(k)\| \Big ) \|x_i(k) - \bar{x}\| \nonumber \\
&+ \|x_i(k) - v(k)\|. \label{eq:bar_v_k_3}
\end{align}
Summing both sides of \eqref{eq:bar_v_k_3} with respect to $i=1,\ldots,m$,
\begin{align}
\sum_{i=1}^m & \|x_i(k) - \bar{v}(k)\| \nonumber \\ & \leq \frac{1}{\rho} \Big ( \sum_{i=1}^m \|x_i(k) - v(k)\| \Big ) \Big ( \sum_{i=1}^m \|x_i(k) - \bar{x}\| \Big ) \nonumber \\ &~~~~~+ \sum_{i=1}^m \|x_i(k) - v(k)\| \nonumber \\
& \leq \Big ( \frac{2}{\rho} m D + 1 \Big ) \sum_{i=1}^m \|x_i(k) - v(k)\|, \label{eq:bar_v_k_4}
\end{align}
where the last inequality is since $\|x_i(k) - \bar{x}\| \leq \|x_i(k)\| + \|\bar{x}\| \leq 2D$ for all $i=1,\ldots,m$, ($D$ as defined above \eqref{eq:Lipschitz}), by Assumption \ref{ass:CompactLip}. This concludes the proof.
\end{proofof}

\begin{proofof}{Lemma \ref{lemma:error_v_k}}
By \eqref{eq:dyn_sys_x}, \eqref{eq:dyn_sys_v}, for all $k,s$ with $s \geq 0$, $k > s$, and for all $i=1,\ldots,m$ we have that
\begin{align}
\|x_i(k+1) &- v(k+1)\| = \Big | \Big | \sum_{j=1}^m \Big ( \big [ \Phi(k,s) \big ]_j^i - \frac{1}{m} \Big ) x_j(s) \nonumber \\
& + \sum_{r=s}^{k-1} \sum_{j=1}^m \Big ( \big [ \Phi(k,r+1) \big ]_j^i - \frac{1}{m} \Big ) e_j(r+1) \nonumber \\
& + e_i(k+1) - \frac{1}{m} \sum_{j=1}^m e_j(k+1) \Big | \Big | \nonumber \\
\leq \sum_{j=1}^m \Big | \big [\Phi(k&,s) \big ]_j^i - \frac{1}{m} \Big | \|x_j(s)\| \nonumber \\
& + \sum_{r=s}^{k-1} \sum_{j=1}^m \Big | \big [\Phi(k,r+1) \big ]_j^i - \frac{1}{m} \Big | \|e_j(r+1)\| \nonumber \\
& + \|e_i(k+1)\| + \frac{1}{m} \sum_{j=1}^m \|e_j(k+1)\|. \label{eq:error_v_1}
\end{align}

Under Assumptions \ref{ass:Network} and \ref{ass:Weights}, by Lemma 4 of \cite{Nedic_Ozdaglar_2009}, for all $k,s$ with $s \geq 0$, $k \geq s$ we have that
\begin{align}
\Big | \big [\Phi(k,s) \big ]_j^i &- \frac{1}{m} \Big | \nonumber \\
& \leq 2 \frac{1 + \eta^{-(m-1)T}}{1 - \eta^{(m-1)T}} \big ( 1 - \eta^{(m-1)T} \big )^{\frac{k-s}{(m-1)T}}. \label{eq:bound_tr}
\end{align}
Setting $\lambda = 2 \big ( 1 + \eta^{-(m-1)T} \big ) / \big ( 1 - \eta^{(m-1)T} \big )$ and $q = \big ( 1 - \eta^{(m-1)T} \big )^{\frac{1}{(m-1)T}}$, \eqref{eq:bound_tr} implies that $\Big | \big [\Phi(k,s) \big ]_j^i - \frac{1}{m} \Big | \leq \lambda q^{k-s}$, for all $k \geq s$. Noticing that $q \in (0,1)$, since $\eta \in (0,1)$, \eqref{eq:error_v_1} and \eqref{eq:bound_tr} lead to \eqref{eq:error_v_k}, thus concluding the proof.
\end{proofof}

\begin{proofof}{Lemma \ref{lemma:sum_error}}
Fix any $N \in \mathbb{N}_+$ and, under Assumptions \ref{ass:Convex} - \ref{ass:Weights}, consider \eqref{eq:sum_error_all}. To show \eqref{eq:sum_e_bound}, we treat each of the three terms in the right-hand side of \eqref{eq:sum_error_all} separately.
\begin{term} \label{term:first}
$2 m \mu \lambda \bar{L} \sum_{k=1}^N  c(k) q^{k} \sum_{i=1}^m \|x_i(0)\|$.
\end{term}
\noindent Due to Assumption \ref{ass:CompactLip}, $\|x_i(0)\| \leq D$, for all $i=1,\ldots,m$. Therefore, $\sum_{i=1}^m \|x_i(0)\| \leq mD$. The last statement together with the fact that, under Assumption \ref{ass:ConvCoef}, $c(k) \leq c(1)$, leads to
\begin{align}
2 m \mu \lambda \bar{L} \sum_{k=1}^N  c(k) q^{k}  \sum_{i=1}^m \|x_i(0)\| & \leq 2 m^2 \mu \lambda \bar{L} D c(1) \sum_{k=1}^N q^{k} \nonumber \\
& <  \frac{2 m^2 \mu \lambda \bar{L} Dc(1) q}{1-q}, \label{eq:sum_termA_1}
\end{align}
where the last step is due to the fact that, by Lemma \ref{lemma:error_v_k}, $q \in (0,1)$ and hence $\sum_{k=1}^{\infty} q^{k} = q \sum_{k=0}^{\infty} q^{k} = q / (1-q)$.

\begin{term} \label{term:second}
$2 m \mu \lambda \bar{L} \sum_{k=1}^N \sum_{r=0}^{k-1} c(k) q^{k-r-1} \sum_{i=1}^m \|e_i(r+1)\|$.
\end{term}
\noindent Fix any $\alpha_1 \in (0,1)$. We then have that
\begin{align}
2 &m \mu \lambda \bar{L} \sum_{k=1}^N \sum_{r=0}^{k-1} c(k) q^{k-r-1} \sum_{i=1}^m \|e_i(r+1)\| \nonumber \\
& = \sum_{i=1}^m  \sum_{k=1}^N \sum_{r=0}^{k-1} 2 \Big ( m \mu \lambda \bar{L} \sqrt{\frac{2}{\alpha_1 (1-q)}} c(k) \Big ) \nonumber \\
& ~~~~~~~~~~~~~~~~ \Big ( \sqrt{\frac{\alpha_1 (1-q)}{2}} \|e_i(r+1)\| \Big ) q^{k-r-1} \nonumber \\
& \leq \sum_{i=1}^m  \sum_{k=1}^N \sum_{r=0}^{k-1} m^2 \mu^2 \lambda^2 \bar{L}^2 \frac{2}{\alpha_1 (1-q)} q^{k-r-1} c(k)^2 \nonumber \\
& + \sum_{i=1}^m  \sum_{k=1}^N \sum_{r=0}^{k-1} \frac{\alpha_1 (1-q)}{2} q^{k-r-1} \|e_i(r+1)\|^2, \label{eq:sum_termB_1}
\end{align}
where in the last step we used the fact that $2xy \leq x^2 + y^2$ for all $x,y \in \mathbb{R}$.
We have that,
\begin{align}
\sum_{k=1}^N &\sum_{r=0}^{k-1} q^{k-r-1} c(k)^2 \nonumber \\
&\leq \sum_{k=1}^N \sum_{r=0}^{k-1} q^{k-r-1} c(r)^2 = \sum_{r=0}^{N-1} c(r)^2 \sum_{t=0}^{N-r-1} q^t \nonumber \\
& <  \sum_{r=0}^{N-1} c(r)^2 \sum_{t=0}^{\infty} q^t = \sum_{k=0}^{N-1} \frac{1}{1-q} c(k)^2 \nonumber \\
& < \frac{1}{1-q} c(0)^2 + \sum_{k=1}^{N} \frac{1}{1-q} c(k)^2, \label{eq:sum_termB_2}
\end{align}
where the first inequality is due to the fact that, under Assumption \ref{ass:ConvCoef}, $c(k) \leq c(r)$ since $k > r$. The first equality is due to series convolution, in the last equality we performed an index change from $r$ to $k$ and the last inequality is included to introduce the desired summation limits.

Repeating the same derivation as in \eqref{eq:sum_termB_2} with $\|e_i(r+1)\|^2$ in place of $c(r)^2$ leads to
\begin{align}
\sum_{k=1}^N \sum_{r=0}^{k-1} &q^{k-r-1} \|e_i(r+1)\|^2 \nonumber \\
& < \frac{1}{1-q} \|e_i(1)\|^2 + \sum_{k=1}^{N} \frac{1}{1-q} \|e_i(k+1)\|^2 \nonumber \\
& \leq \frac{4}{1-q} D^2 + \sum_{k=1}^{N} \frac{1}{1-q} \|e_i(k+1)\|^2, \label{eq:sum_termB_3}
\end{align}
where the last inequality is due to the fact that $\|e_i(1)\| \leq 2D$ under Assumption \ref{ass:CompactLip}.

By \eqref{eq:sum_termB_1}, \eqref{eq:sum_termB_2}, \eqref{eq:sum_termB_3}, and noticing that some terms are independent of $i$, we have that
\begin{align}
2 m \mu \lambda \bar{L} &\sum_{k=1}^N \sum_{r=0}^{k-1} c(k) q^{k-r-1} \sum_{i=1}^m \|e_i(r+1)\| \nonumber \\
& < 2 m^3 \mu^2 \lambda^2 \bar{L}^2 \frac{1}{\alpha_1 (1-q)^2} c(0)^2 + 2 \alpha_1 m D^2 \nonumber \\
& + 2 m^3 \mu^2 \lambda^2 \bar{L}^2 \frac{1}{\alpha_1 (1-q)^2} \sum_{k=1}^N c(k)^2 \nonumber
\end{align}
\begin{align}
& + \frac{\alpha_1}{2} \sum_{k=1}^N \sum_{i=1}^m \|e_i(k+1)\|^2, \label{eq:sum_termB_4}
\end{align}

\begin{term} \label{term:third}
$4 \mu \bar{L} \sum_{k=1}^N c(k) \sum_{i=1}^m \|e_i(k+1)\|$.
\end{term}
\noindent We have that
\begin{align}
4 &\mu \bar{L} \sum_{k=1}^N c(k) \sum_{i=1}^m \|e_i(k+1)\| \nonumber \\
& = \sum_{k=1}^N \sum_{i=1}^m 2 \Big ( 2 \sqrt{\frac{2}{\alpha_1}} \mu \bar{L} c(k) \Big ) \Big ( \sqrt{\frac{\alpha_1}{2}} \|e_i(k+1)\| \Big ) \nonumber \\
& \leq  \sum_{k=1}^N \sum_{i=1}^m \frac{8}{\alpha_1} \mu^2 \bar{L}^2 c(k)^2 + \frac{\alpha_1}{2} \sum_{k=1}^N \sum_{i=1}^m \|e_i(k+1)\|^2 \nonumber \\
& = \frac{8}{\alpha_1} m \mu^2 \bar{L}^2 \sum_{k=1}^N c(k)^2 + \frac{\alpha_1}{2} \sum_{k=1}^N \sum_{i=1}^m \|e_i(k+1)\|^2, \label{eq:sum_termC_1}
\end{align}
where for the first inequality we follow the same reasoning with the last step of \eqref{eq:sum_termB_1}, and the second equality is since the first term of the first inequality is independent of $i$.

We are now in a position to show \eqref{eq:sum_e_bound}.
Substituting \eqref{eq:sum_termA_1}, \eqref{eq:sum_termB_4} and \eqref{eq:sum_termC_1} in \eqref{eq:sum_error_all}, and setting $\alpha_2, \alpha_3$ according to \eqref{eq:constants_lm}, leads to
\eqref{eq:sum_e_bound} (the inequality is strict since the inequalities in \eqref{eq:sum_termA_1}, \eqref{eq:sum_termB_4} are also strict), thus concluding the proof.
\end{proofof}

\end{document}